\documentclass[11pt,a4paper]{amsart}
\usepackage[utf8]{inputenc}
\usepackage{amssymb}
\usepackage{amsmath}
\usepackage{amsthm}
\usepackage{mathtools}
\usepackage{stmaryrd}
\usepackage{xcolor}
\usepackage{hyperref}
    \hypersetup{colorlinks=true,allcolors=red}
    \usepackage[capitalize,nameinlink,noabbrev]{cleveref}
\usepackage{tikz-cd}
\usepackage[top=3 cm, bottom=3 cm, left=3 cm, right=3cm]{geometry}
\usepackage{epigraph}
\usepackage[backref=true,giveninits=true,doi=false,url=false,isbn=false,backend=biber]{biblatex}
    \newbibmacro{string+doiurlisbn}[1]{%
      \iffieldundef{doi}{%
        \iffieldundef{url}{%
          \iffieldundef{isbn}{%
            \iffieldundef{issn}{%
              #1%
            }{%
              \href{http://books.google.com/books?vid=ISSN\thefield{issn}}{#1}%
            }%
          }{%
            \href{http://books.google.com/books?vid=ISBN\thefield{isbn}}{#1}%
          }%
        }{%
          \href{\thefield{url}}{#1}%
        }%
      }{%
        \href{http://dx.doi.org/\thefield{doi}}{#1}%
      }%
    }
    \DeclareFieldFormat{title}{\usebibmacro{string+doiurlisbn}{\mkbibemph{#1}}}
    \DeclareFieldFormat[article,incollection,thesis,misc,inproceedings,online,inbook]{title}{\usebibmacro{string+doiurlisbn}{\mkbibquote{#1}}}
    \DeclareFieldFormat{year}{\usebibmacro{year-parenthesis}{\mkbibemph{#1}}}
\renewbibmacro*{date}{%
\iffieldundef{year}
{}
{\ifentrytype{misc}{\printtext[parens]{\printdate}}{\printdate}}}%
\hypersetup{
           breaklinks=true,   
           colorlinks=true,   
           pdfusetitle=true,  
        }

\makeatletter
\let\@wraptoccontribs\wraptoccontribs
\makeatother

\newcounter{dummy}
\numberwithin{dummy}{section}
\newtheorem{thm}[dummy]{Theorem}
\newtheorem{defn}[dummy]{Definition}
\newtheorem{conj}[dummy]{Conjecture}
\newtheorem{lem}[dummy]{Lemma}
\newtheorem{prop}[dummy]{Proposition}
\newtheorem{cor}[dummy]{Corollary}
\newtheorem{ex}[dummy]{Example}
\theoremstyle{definition}
\newtheorem{rmk}[dummy]{Remark}

\numberwithin{equation}{section}

\DeclareMathOperator{\codim}{codim}
\DeclareMathOperator{\End}{End}
\DeclareMathOperator{\GL}{GL}

\DeclareMathOperator{\id}{id}
\DeclareMathOperator{\Mat}{Mat}
\DeclareMathOperator{\mult}{mult}

\DeclareMathOperator{\ord}{ord}

\DeclareMathOperator{\red}{red}
\DeclareMathOperator{\Spec}{Spec}

\title{Arithmetic unlikely intersections in powers of the multiplicative group}

\date{\today}

\author[F. Campagna]{Francesco Campagna}
\address[F. Campagna]{Université Clermont Auvergne - LMBP UMR 6620 - CNRS, Campus des Cézeaux 3, place Vasarely 63178 Aubière cedex, France}
\email{francesco.campagna@uca.fr}

\author[G. A. Dill]{Gabriel A. Dill}
\address[G. A. Dill]{Institut de Math\'ematiques, Universit\'e de Neuch\^atel, Rue Emile-Argand 11, 2000 Neuch\^atel, Switzerland}
\email{gabriel.dill@unine.ch}

\address[R. Wilms]{Universit\'e de Caen Normandie, CNRS, LMNO UMR 6139, F-14000 Caen, France}
\email{robert.wilms87@gmail.com}

\contrib[with an appendix joint with]{Robert Wilms}

\keywords{Arithmetic B\'ezout theorem, arithmetic intersection theory, divisibility sequence, greatest common divisor, multiplicative dependence.}
\subjclass{11G50, 14G40.}

\begin{document}

\maketitle
\begin{abstract}
Inspired by work of Bugeaud-Corvaja-Zannier, we formulate a conjecture about unlikely intersections in powers of the multiplicative group over the ring of integers in a number field. Broadly speaking, if an intersection with a subgroup scheme is unlikely for dimension reasons, its ``size" should not be too big compared to the ``complexity" of the subgroup scheme. We first obtain some results on likely intersections that serve as a benchmark for the unlikely case and generalize work of Barroero-Capuano-M\'erai-Ostafe-Sha. We then show that our conjecture in dimension $1$ follows from work of Corvaja-Zannier, we obtain some partial result in dimension $2$, and we present some open problems that are special cases of the conjecture.
\end{abstract}

\vspace{0.3cm}

\epigraph{Ich setze im Ich dem teilbaren Ich ein teilbares
Nicht-Ich entgegen.}{J. G. Fichte, Grundlage der gesammten Wissenschaftslehre}

\tableofcontents

\section{Introduction}

For $N \in \mathbb{N}=\{1,2,...\}$, let $a_N$ be the cardinality of the finite ring 
\[
A_N:= \mathbb{Z}[T]/(T^N-1, (T+1)^N-1, (T-1)^N-1).
\]
For $N\mid M$, there is a natural projection map $A_M \twoheadrightarrow A_N$, so the sequence $\{a_N\}_{N \in \mathbb{N}}$ is a divisibility sequence whose first terms are
\[
a_1=1, \hspace{0.2cm} a_2=1, \hspace{0.2cm} a_3=4, \hspace{0.2cm} a_4=25=5^2, \hspace{0.2cm} a_5=11, \hspace{0.2cm} a_6=153664=2^6\cdot 7^4.
\]

At first sight, the sequence $\{a_N\}_{N\in\mathbb N}$ appears rather mysterious. Its first values do not seem to exhibit any obvious pattern, and, to the best of our knowledge, this sequence has not been studied before. Nevertheless, we shall argue that one should expect a strong asymptotic constraint on its growth, namely
\begin{equation} \label{eq:lim_to_0}
    \lim_{N\to\infty} \frac{\log a_N}{N^2}=0.
\end{equation}
The point of this article is that this prediction does not arise from a direct analysis of the rings $A_N$, but rather from a more conceptual geometric framework. Indeed, the rings $A_N$ admit a natural interpretation as arithmetic intersection rings, and the above asymptotic behaviour is predicted by a general philosophy of \textit{arithmetic unlikely intersections} that we shall develop here.

In characteristic 0, the theory of unlikely intersections, \emph{i.e.} of intersections that are unexpected for dimension reasons, has a rich history, beginning with the pioneering work of Bombieri, Masser, and Zannier \cite{BMZ99}. The central conjectures in the field have been formulated by Zilber \cite{Zilber}, Pink \cite{PinkUnpubl}, and Bombieri, Masser, and Zannier \cite{BMZ07}. For a comprehensive account, we refer to the books of Zannier \cite{ZannierBook} and of Pila \cite{Pila_book}.

On the other hand, many statements about unlikely intersections in characteristic $0$ become hopelessly false in mixed or positive characteristic, already in relatively simple cases like that of the Manin--Mumford conjecture (every point defined over a finite field is torsion).

In positive characteristic, one usually considers varieties defined over function fields and excludes varieties that are defined over finite fields, see \emph{e.g.} \cite{Pink_Roessler_2004,  Masser_2014,Brownawell_Masser_2017}. For Manin-Mumford and Andr\'e-Oort-like statements over finite fields as well as in mixed characteristic, see \cite{Richard_2018,Edixhoven_Richard_2020,Baldi_Richard_Ullmo_2021}.

In this article, we study unlikely intersections in mixed characteristic from a different angle. Our approach is inspired by the following theorem by Bugeaud, Corvaja, and Zannier \cite[Theorem 1]{Bugeaud_Corvaja_Zannier_2003}, which we use to illustrate the general framework in which we will work: let $a,b \in \mathbb{N}$ be multiplicatively independent integers, and let $\varepsilon>0$. Then, provided $N \in \mathbb{N}$ is sufficiently large, we have
\begin{equation} \label{eq:BCZ_theorem}
    \log \gcd(a^N-1,b^N-1) < \varepsilon N.
\end{equation}

We do not know whether this upper bound is best possible. It is known that, for infinitely many $N \in \mathbb{N}$, the left-hand side of \eqref{eq:BCZ_theorem} is bounded from below by $N^{c/\log \log N}$ for some $c > 0$; this is mentioned in \cite{Bugeaud_Corvaja_Zannier_2003} to follow from work of Adleman, Pomerance, and Rumely \cite[Proposition 10]{Adleman_Pomerance_Rumely_1983}. For more results about g.c.d. sequences, see for instance \cite{AilonRudnick, Silverman_2004, Silverman_2005, Silverman_2017, Levin_2019, Barroero_Capuano_Turchet}.

As already suggested by Zannier in \cite[Chapter 2]{ZannierBook}, the inequality \eqref{eq:BCZ_theorem} admits a natural interpretation in terms of arithmetic unlikely intersections inside $\mathbb{G}_{m,\mathbb{Z}}^2$, the square of the arithmetic multiplicative group over $\mathbb{Z}$. Indeed, let $\mathcal{V}\subseteq \mathbb{G}_{m,\mathbb{Z}}^2$ be the Zariski closure of the rational point $(a,b) \in \mathbb{G}_{m,\mathbb{Q}}^2(\mathbb{Q})$. Then $\mathcal{V}$ is an arithmetic scheme of (absolute) dimension 1 and the condition on the multiplicative independence of $a$ and $b$ is equivalent to the statement that the generic fiber of $\mathcal{V}$ is not contained in any proper algebraic subgroup of $\mathbb{G}_{m,\mathbb{Q}}^2$. This in particular implies that for all $N \in \mathbb{N}$, the scheme-theoretic intersection of $\mathcal{V}$ with the kernel $\mathcal{H}_N = \ker [N]$ of the $N$-th power map on $\mathbb{G}_{m,\mathbb{Z}}^2$ is zero-dimensional. Note that $\dim \mathcal{H}_N=1$ so that $\dim \mathcal{V} + \dim \mathcal{H}_N < 3 = \dim \mathbb{G}_{m,\mathbb{Z}}^2$ and the intersection is unlikely. A computation shows that $\mathcal{V} \cap \mathcal{H}_N=\Spec R_N$, where
\[
R_N=\mathbb{Z}/\gcd(a^N-1,b^N-1)\mathbb{Z},
\]
so the left-hand side of \eqref{eq:BCZ_theorem} is equal to $\log |R_N|$. This quantity may be regarded as a measure of the ``arithmetic size" of the intersection scheme $\mathcal{V} \cap \mathcal{H}_N$. On the other hand, the factor $N$ on the right-hand side can be interpreted as the ``complexity" of the subgroup scheme $\mathcal{H}_N$. Thus, at an informal level, inequality \eqref{eq:BCZ_theorem} asserts that the arithmetic size of the unlikely intersection $\mathcal{V}\cap \mathcal{H}_N$ is ``small" compared to the complexity of $\mathcal{H}_N$.

The aim of this article is to extend this picture to a broader setting, namely to certain arithmetic unlikely intersections in $\mathbb{G}^n_{m,\mathcal{O}_K}$, where $\mathcal{O}_K$ is the ring of integers of a fixed number field $K$ and $n\in \mathbb{N}$. We formulate a conjecture on arithmetic \textit{likely intersections} in this context (Conjecture \ref{conj:LIB}) as well as a companion conjecture on arithmetic \textit{unlikely intersections} (Conjecture \ref{conj:mainconj}). We present several results supporting these conjectures, some of a general nature and others arising from specific families of examples. The rest of this introduction contains a more detailed overview of the contents of our paper that we have now just briefly described.

After fixing notation and recalling some preliminaries in Section \ref{sec:preliminaries}, we begin in Section \ref{sec:complexity} by defining the complexity $\mathcal{C}(\mathcal{H})$ of certain closed subgroup schemes $\mathcal{H} \subseteq \mathbb{G}^n_{m,\mathcal{O}_K}$, see Definition \ref{defn:complexity}. The closed subgroup schemes $\mathcal{H}$ for which we define the complexity are the kernels of finite sets of endomorphisms of $\mathbb{G}^n_{m,\mathcal{O}_K}$. These are precisely the closed subgroup schemes of $\mathbb{G}^n_{m,\mathcal{O}_K}$ that are flat over $\Spec \mathcal{O}_K$ as we show in Proposition \ref{prop:complexityflat}. With our definition, the kernel of the $N$-th power map on $\mathbb{G}^n_{m,\mathcal{O}_K}$ has complexity $N$ for all $N \in \mathbb{N}$, see Lemma \ref{lem:complexitymultbyN}.

In Section \ref{sec:trivialbound}, we apply the arithmetic intersection theory of Bost, Gillet, and Soul\'e \cite{BGS} to prove the following result about (possibly) likely intersections in $\mathbb{G}^n_{m,\mathcal{O}_K}$:

\begin{thm}[Theorem \ref{thm:LIB}]\label{thm:LIBintro}
Let $\mathcal{V} \subseteq \mathbb{G}^n_{m,\mathcal{O}_K}$ be an integral closed subscheme that dominates $\Spec \mathcal{O}_K$. There exists a constant $C = C(n,K,\mathcal{V}) \in \mathbb{R}$ such that for all flat subgroup schemes $\mathcal{H} \subseteq \mathbb{G}^n_{m,\mathcal{O}_K}$ which satisfy $\dim \mathcal{H} + \dim \mathcal{V} \leq n+1$, we have
\begin{equation}
\sum_{\mathcal{Z}}{\deg(\mathcal{Z}_{\red})\log N(\mathcal{Z})} \leq C\mathcal{C}(\mathcal{H})^{\dim \mathcal{V}},
\end{equation}
where the sum on the left-hand side runs over the irreducible components of $\mathcal{V} \cap \mathcal{H}$ which do not dominate $\Spec \mathcal{O}_K$, and for each such component $\mathcal{Z}$ we denote by $N(\mathcal{Z})$ the norm of the maximal ideal $\mathfrak{P}\subseteq \mathcal{O}_K$ such that $\mathcal{Z} \subseteq \mathbb{G}^n_{m,\mathbb{F}_{\mathfrak{P}}}$, and by $\deg(\mathcal{Z}_{\red})$ the degree of the Zariski closure of $\mathcal{Z}_{\red}$ in $\mathbb{P}^n_{\mathbb{F}_{\mathfrak{P}}}$.
\end{thm}

We exclude the case where $\dim \mathcal{H} + \dim \mathcal{V} > n+1$ since in that case, we actually expect the generic fiber of the intersection to be non-empty and so the problem's arithmetic nature retreats somewhat to the background.

Theorem \ref{thm:LIBintro} can be seen as a likely intersection bound \textit{without multiplicities}. For instance, if $n=1$, $\mathcal{V}$ is the Zariski closure in $\mathbb{G}_{m,\mathbb{Z}}$ of some point $a \in \mathbb{G}_{m,\mathbb{Q}}(\mathbb{Q}) \backslash \{1,-1\}$, and $\mathcal{H}=\mathcal{H}_N = \ker [N]$ is taken as in the theorem of Bugeaud, Corvaja and Zannier mentioned above, our Theorem \ref{thm:LIBintro} implies a bound of the form
\[
\log \mathrm{rad} (a^N-1) \leq C N,
\]
where $\mathrm{rad}(\cdot)$ denotes the radical of an integer and $C>0$. Establishing such a bound certainly does not require our theorem, even without taking the radical. The strength of Theorem \ref{thm:LIBintro} resides rather in its generality as well as in the powerful methods underlying its proof. Indeed, these methods allow us to improve on a result of Barroero, Capuano, M\'erai, Ostafe, and Sha \cite[Theorem 2.1]{BCMOS_Preprint} on multiplicative dependence with two independent relations modulo primes.

\begin{thm}[Theorem \ref{cor:curveingm}]\label{cor:curveingmINTRO}
Let $\mathcal{V} \subseteq \mathbb{G}^{n}_{m,\mathcal{O}_K}$ be a $2$-dimensional integral closed subscheme that dominates $\Spec \mathcal{O}_K$. Suppose that the generic fiber of $\mathcal{V}$ is not contained in any proper algebraic subgroup of $\mathbb{G}^{n}_{m,K}$.

There exist a finite set $S = S(\mathcal{V})$ of $1$-dimensional integral closed subschemes $\mathcal{X} \subseteq \mathcal{V}$, flat over $\Spec \mathcal{O}_K$, and a constant $C = C(K,n,\mathcal{V})$ such that the following holds: let $(k_1,\hdots,k_n)$ and $(l_1,\hdots,l_n)$ be two linearly independent integer vectors, set $M = (\max_{i}{|k_i|})(\max_{i}{|l_i|})$, and let $\mathcal{H} \subseteq \mathbb{G}^n_{m,\mathcal{O}_K}$ be the subgroup scheme defined by the equations $\prod_{i=1}^{n}{X_i^{k_i}} = \prod_{i=1}^{n}{X_i^{l_i}} = 1$ in the affine coordinates $X_1,\hdots,X_n$ on $\mathbb{G}^n_{m,\mathcal{O}_K}$.

Then every irreducible component $\mathcal{Z}$ of $\mathcal{V} \cap \mathcal{H}$ that dominates $\Spec \mathcal{O}_K$ belongs to $S$ and furthermore
\[ \sum_{\mathcal{Z}}{\deg(\mathcal{Z}_{\red})\log N(\mathcal{Z})} \leq CM,\]
where the sum runs over all irreducible components $\mathcal{Z}$ of $\mathcal{V} \cap \mathcal{H}$ that are not contained in an element of $S$.
\end{thm}

We refer the reader to Section \ref{subsec:BCMOS} for the proof of Theorem \ref{cor:curveingmINTRO} and a thorough comparison with \cite[Theorem 2.1]{BCMOS_Preprint}.

As already hinted in the discussion above, in presence of a likely arithmetic intersection, we expect a general likely intersection bound \textit{with multiplicities}.

\begin{conj}[Conjecture \ref{conj:LIB}] \label{conj:LIB_INTRO}
Let $\mathcal{V} \subseteq \mathbb{G}^n_{m,\mathcal{O}_K}$ be an integral closed subscheme that dominates $\Spec \mathcal{O}_K$. There exists a constant $C = C(n,K,\mathcal{V}) \in \mathbb{R}$ such that for all flat subgroup schemes $\mathcal{H} \subseteq \mathbb{G}^n_{m,\mathcal{O}_K}$ which satisfy $\dim \mathcal{H} + \dim \mathcal{V} \leq n+1$, we have
\begin{equation}\label{eq:trivialbound2INTRO}
\sum_{\mathcal{Z}}{\mult(\mathcal{Z})\deg(\mathcal{Z}_{\red})\log N(\mathcal{Z})} \leq C\mathcal{C}(\mathcal{H})^{\dim \mathcal{V}},
\end{equation}
where the sum on the left-hand side runs over the irreducible components of $\mathcal{V} \cap \mathcal{H}$ which do not dominate $\Spec \mathcal{O}_K$ and for each such component $\mathcal{Z}$ we denote by $\mult(\mathcal{Z})$ its multiplicity, which we shall define below in Definition \ref{defn:multiplicity}.
\end{conj}

The left-hand side of \eqref{eq:trivialbound2INTRO} should be regarded as the true ``size" of the intersection; if the intersection is zero-dimensional, then it is isomorphic to $\Spec R$ for some finite ring $R$ and we will show in Remark \ref{rmk:lhszerodim} that this left-hand side is exactly $\log |R|$. We do not know how to prove Conjecture \ref{conj:LIB_INTRO} in full generality although, if $\dim \mathcal{V} = 1$, one can obtain a proof following along the lines of the proof of our Theorem \ref{thm:ULIB_dim1}, see Remark \ref{rmk:LIB_dim1}. In Section \ref{subsec:trivialboundex}, we prove Conjecture \ref{conj:LIB_INTRO} for a certain family of examples in $\mathbb{G}^2_{m,\mathbb{Z}}$ with $\mathcal{H} = \mathcal{H}_N$ the kernel of the $N$-th power map ($N \in \mathbb{N}$), see Proposition \ref{prop:LIBexample}. In the special case where $\mathcal{V}$ is defined by $Y = X+1$ and under the hypothesis that $6 \nmid N$, we deduce from the work of Dimitrov-Habegger \cite{Dimitrov_Habegger_2019} an asymptotic for the left-hand side of \eqref{eq:trivialbound2INTRO} that is of the same order of magnitude as the upper bound, see Proposition \ref{prop:Lfunction}.

In Section \ref{sec:mainconj}, we turn to the study of arithmetic \textit{unlikely} intersections. The bound \eqref{eq:BCZ_theorem} leads us to expect that, if the intersection is unlikely for dimension reasons and the generic fiber of $\mathcal{V}$ is not contained in any proper algebraic subgroup, we should be able to replace the constant $C$ on the right-hand side of \eqref{eq:trivialbound2INTRO} by an arbitrarily small $\varepsilon > 0$ at the cost of adding another constant.

\begin{conj}[Conjecture \ref{conj:mainconj}] \label{conj:mainconjINTRO}
Let $\mathcal{V} \subseteq \mathbb{G}_{m,\mathcal{O}_K}^n$ be an integral closed subscheme that dominates $\Spec \mathcal{O}_K$. Suppose that the generic fiber of $\mathcal{V}$ is not contained in any proper algebraic subgroup of $\mathbb{G}^n_{m,K}$. For every $\varepsilon > 0$, there exists a constant $c = c(n,\mathcal{V},\varepsilon, K) \in \mathbb{R}$ such that for all flat subgroup schemes $\mathcal{H} \subseteq \mathbb{G}_{m,\mathcal{O}_K}^n$ which satisfy $\dim \mathcal{H} + \dim \mathcal{V} < n+1$ we have

\begin{equation}\label{eq:inequalityconjectureINTRO}
    \sum_{\mathcal{Z}}{\mult(\mathcal{Z})\deg(\mathcal{Z}_{\mathrm{\red}})\log N(\mathcal{Z})} \leq \varepsilon\mathcal{C}(\mathcal{H})^{\dim \mathcal{V}} + c,
\end{equation}
where the sum runs over all irreducible components $\mathcal{Z}$ of $\mathcal{V} \cap \mathcal{H}$ which do not dominate $\Spec \mathcal{O}_K$.
\end{conj}

The case of arithmetic unlikely intersections seems to be significantly more challenging. In the case $\dim \mathcal{V}=1$, we show in Theorem \ref{thm:ULIB_dim1} that Conjecture \ref{conj:mainconjINTRO} holds. Our proof heavily relies on work of Corvaja and Zannier, in particular \cite{Corvaja_Zannier_2005}. In the case where $\dim \mathcal{V}=2$ and $\mathcal{H}$ is the kernel of the $N$-th power map for some $N \in \mathbb{N}$, we manage to obtain bounds on the characteristic of the fibers at which the intersection $\mathcal{V}\cap \mathcal{H}$ is non-empty. Our theorem includes the likely case where $n = 2$; in this special case, we get more precise information than from Theorem \ref{thm:LIBintro} and we improve on the work of Barroero, Capuano, M\'erai, Ostafe, and Sha \cite{BCMOS_Preprint}.

\begin{thm}[Theorem \ref{thm:singleprimecurve}]\label{thm:singleprimecurveINTRO}
Suppose that $n \geq 2$. Let $\mathcal{V} \subseteq \mathbb{G}^n_{m,\mathcal{O}_K}$ be an integral closed subscheme of dimension $2$ that dominates $\Spec \mathcal{O}_K$. Suppose that the generic fiber $\mathcal{V}_K$ of $\mathcal{V}$ is not contained in any proper algebraic subgroup of $\mathbb{G}^n_{m,K}$. For a maximal ideal $\mathfrak{P}$ of $\mathcal{O}_K$, we denote by $N(\mathfrak{P})$ its norm and by $\mathcal{V}_{\mathfrak{P}}$ the fiber of $\mathcal{V}$ over $\mathfrak{P}$. Then, there exists $C = C(K,n,\mathcal{V})$ such that the following holds:

Suppose that $N \in \mathbb{N}$, $\mathfrak{P}$ is a maximal ideal of $\mathcal{O}_K$, and $x \in \mathcal{V}_{\mathfrak{P}} \cap \ker [N]$. Then one of the following holds:
\begin{enumerate}
    \item $\log N(\mathfrak{P}) \leq CN^{\frac{2n-1}{n(n-1)}}$, or
    \item there exists $y \in \mathcal{V}_K$ such that $y$ is contained in an algebraic subgroup of $\mathbb{G}^n_{m,K}$ of codimension at least $2$ and $x$ is contained in the Zariski closure of $y$ in $\mathcal{V}$.
\end{enumerate}

In both cases, we have that
\[\log N(\mathfrak{P}) \leq \begin{cases} CN^{\frac{3}{2}} &\text{if } n = 2,\\
CN &\text{if } n \geq 3,
\end{cases}\]
unless $y$ can be chosen inside $\ker [N]$.
\end{thm}

If $n \geq 3$ and we are in case (2), then the upper bound in Theorem \ref{thm:singleprimecurveINTRO} is probably optimal, see the discussion after Theorem \ref{thm:singleprimecurve}.

We conclude in Section \ref{sec:openproblems} by presenting two open problems that are motivated by Conjecture \ref{conj:mainconjINTRO} and whose solutions are predicted by this same conjecture. While both problems are easy to state, they both appear quite challenging to us. One of these problems (Problem 1) concerns the growth of the divisibility sequence $\{a_N\}_{N\in\mathbb N}$ that appeared at the beginning of this introduction. This sequence arises from the unlikely intersection in $\mathbb{G}_{m,\mathbb{Z}}^3$ of a 2-dimensional subscheme with the kernel of the $N$-th power map. Consequently, Conjecture \ref{conj:mainconjINTRO} predicts precisely that the limit in \eqref{eq:lim_to_0} should vanish. In Appendix \ref{appendix:data}, which is co-authored with Robert Wilms, we present some computational data that even seem to support the stronger claim $\log a_N = \mathcal{O}(N \log N)$ as $N \to \infty$, which would be best possible (see the discussion of Problem 1 in Section \ref{sec:openproblems}).

Considered in isolation, these questions may look like somewhat ad hoc arithmetic problems. One of the messages of this article is that they arise naturally from a broader theory of arithmetic unlikely intersections, and we hope that this perspective makes them appear to the reader as compelling as they seem to us.

\section{Preliminaries} \label{sec:preliminaries}

\subsection{Generalities}\label{subsec:generalities}

Let $\mathcal{O}$ be a Dedekind domain with fraction field $F$. For a non-zero prime ideal $\pi$ of $\mathcal{O}$ and $\alpha \in F^{\ast}$, we denote by $\ord_\pi(\alpha) \in \mathbb{Z}$ the order with which $\pi$ appears in the prime factorization of $\alpha\mathcal{O}$. If $\alpha_1,\hdots,\alpha_n \in F$, not all zero, and $S$ is a set of elements or ideals of $\mathcal{O}$, then we set
\[\gcd_S(\alpha_1,\hdots,\alpha_n) = \prod_{\pi}{\pi^{\min_{i,\alpha_i \neq 0}\{\ord_\pi(\alpha_i)\}}},\]
where the product runs over all non-zero prime ideals of $\mathcal{O}$ that do not divide any element of $S$. We set $\gcd_S(0,\hdots,0) = 0$.

For a number field $K$, we denote by $\mathcal{O}_K$ its ring of integers and by $\mathcal{O}_{K,S}$ its ring of $S$-integers for a finite set $S$ of maximal ideals of $\mathcal{O}_K$. The norm of a non-zero ideal $I \subseteq \mathcal{O}_{K,S}$ is denoted by $N(I)$ and the residue field of a maximal ideal $\mathfrak{p} \subseteq \mathcal{O}_{K,S}$ is denoted by $\mathbb{F}_{\mathfrak{p}}$. The set of places of $K$ will be denoted by $M_K$. For each $v \in M_K$, we use the associated absolute value $|\cdot|_v$, normalized as in \cite[Section 1.4]{BombieriGubler}. In particular, the absolute logarithmic Weil height of $\alpha \in K$ equals
\[ h(\alpha) = \sum_{v \in M_K}{\log^+|\alpha|_v},\]
where $\log^+ x = \log\max\{x,1\}$ for $x \in \mathbb{R}$.

For a ring $R$, we denote by $\textrm{Mat}_{m \times n}(R)$ the set of $m \times n$-matrices with entries in $R$ for $m,n \in \mathbb{N}$. Suppose now that $R$ is commutative. For $a_1,\hdots,a_n \in R$, we will denote by $(a_1,\hdots,a_n)$ the ideal $a_1R+\cdots+a_nR$ of $R$. The discriminant of a non-zero polynomial $P \in R[T]$ will be denoted by $\mathrm{disc}(P)$. The $n$-th cyclotomic polynomial with integer coefficients will be denoted by $\Psi_n$ for $n \in \mathbb{N}$.

The maximum norm of a matrix $A = (a_{i,j})_{i,j}$ with complex entries is $\lVert A \rVert_{\infty} = \max_{i,j} |a_{i,j}|$. If $A$ is an $m \times m$-matrix with integer entries, we define the \textit{mass} of $A$ as
\[m(A) := \min_{S \in \GL_m(\mathbb{Z})}{\lVert SA \rVert_{\infty}}.\]
The minimum exists since $\lVert SA \rVert_{\infty} \in \mathbb{Z}_{\geq 0}$ for all $S \in \GL_m(\mathbb{Z})$.

\subsection{Algebraic geometry}
If $S$ is any scheme, $S'$, $T$, and $T'$ are $S$-schemes, and $\varphi: T \to T'$ is a morphism of $S$-schemes, then we denote the base change $T \times_S S'$ of $T$ to $S'$ by $T_{S'}$ and we denote the base change $T_{S'} \to T'_{S'}$ of $\varphi$ by $\varphi_{S'}$. We denote projective space over $S$ by $\mathbb{P}^n_S$ or, if $S = \Spec R$ for a commutative ring $R$, by $\mathbb{P}^n_R$. We denote by $S_{\mathrm{red}}$ the scheme with the same underlying topological space as $S$, but endowed with the reduced scheme structure. If $R$ is a commutative ring, we denote by $R_{\mathrm{red}}$ the quotient of $R$ by its nilradical. For us, a variety over a field $F$ is a reduced and separated scheme of finite type over $F$.

We use the dimension of a scheme as defined in \cite[Definition 5.5]{GoertzWedhorn}. The following lemma shows that we have a reasonable theory of dimension and codimension on irreducible schemes that are dominant and of finite type over $\Spec \mathcal{O}_K$ for some number field $K$. Recall that, by \cite[Propositions 14.107 and 14.109]{GoertzWedhorn}, any scheme that is of finite type over a finite-dimensional Noetherian scheme is finite-dimensional as well; in particular, this applies to schemes of finite type over $\Spec \mathcal{O}_K$.

\begin{lem}\label{lem:dimandcodim}
Let $K$ be a number field, let $\mathcal{V}$ be an irreducible scheme, and let $\pi: \mathcal{V} \to \Spec \mathcal{O}_K$ be a dominant morphism of finite type. Then the following hold:
\begin{enumerate}
    \item for every (not necessarily closed) point $p \in \pi(\mathcal{V})$, the fiber $\pi^{-1}(p)$ is equidimensional of dimension $\dim \mathcal{V}-1$,
     \item every non-empty open subscheme $\mathcal{U}$ of $\mathcal{V}$ is irreducible of dimension $\dim \mathcal{V}$, and
     \item for every irreducible closed subscheme $\mathcal{W} \subseteq \mathcal{V}$, we have that $\dim \mathcal{W} + \codim_{\mathcal{V}} \mathcal{W} = \dim \mathcal{V}$.
\end{enumerate}
\end{lem}

\begin{proof}
Let $\eta$ denote the generic point of $\Spec \mathcal{O}_K$. By \cite[Proposition 15.4~(1)]{GoertzWedhorn}, the image of $\mathcal{V}$ under $\pi$ is open in $\Spec \mathcal{O}_K$. It follows from \cite[Proposition 4.16 in Chapter 4]{Liu_book}, applied to $\pi: \mathcal{V} \to \Spec \mathcal{O}_K$, that every irreducible component of $\pi^{-1}(p)$ is of dimension $\dim \mathcal{V}_{\eta}$ for every $p \in \pi(\mathcal{V})$. By \cite[Proposition 14.105]{GoertzWedhorn}, the scheme $\pi(\mathcal{V})$ is universally catenary. Therefore, we can apply \cite[Proposition 14.109~(2)]{GoertzWedhorn} to $\pi: \mathcal{V} \to \pi(\mathcal{V})$ and obtain that $\dim \mathcal{V}_{\eta} = \dim \mathcal{V}-1$. The first assertion follows.

For the second assertion, we note that $\mathcal{U}$ is open in $\mathcal{V}$, so irreducible, and $\dim \mathcal{U} \leq \dim \mathcal{V}$ by \cite[Lemma 5.7~(1)]{GoertzWedhorn}. By \cite[Proposition 10.9]{GoertzWedhorn}, $\mathcal{V}$ is Noetherian and so the open immersion $\mathcal{U} \hookrightarrow \mathcal{V}$ is of finite type by \cite[Corollary 3.22 and Proposition 10.7~(1)]{GoertzWedhorn}. Since $\mathcal{U}$ is non-empty, so dense in $\mathcal{V}$, it follows that $\pi|_\mathcal{U}: \mathcal{U} \to \Spec \mathcal{O}_K$ is dominant and of finite type. We can therefore apply the first assertion to $\mathcal{U}$ and deduce that $\dim \mathcal{U} = \dim \mathcal{U}_{\eta} + 1$. But $\mathcal{U}_{\eta}$ is a non-empty open subset of $\mathcal{V}_{\eta}$, which is irreducible since $\mathcal{V}$ is irreducible, and hence $\dim \mathcal{U}_{\eta} = \dim \mathcal{V}_{\eta} = \dim \mathcal{V}-1$ by the first assertion together with \cite [Theorem 5.22~(3)]{GoertzWedhorn}. The second assertion follows.

For the third assertion, we let $\eta_{\mathcal{W}}$ denote the generic point of $\mathcal{W}$. We distinguish two cases following \cite[Proposition 15.4]{GoertzWedhorn}. First, $\mathcal{W}$ can be contained (as a set) in $\pi^{-1}(p)$ for some closed point $p \in \Spec \mathcal{O}_K$. In that case, the first assertion and \cite[Proposition 5.30]{GoertzWedhorn} imply that $\dim \mathcal{W} + \codim_{\mathcal{V}'}{\mathcal{W}} = \dim \mathcal{V}' = \dim \mathcal{V}-1$ for some irreducible component $\mathcal{V}'$ of $\pi^{-1}(p)$ containing $\mathcal{W}$. If $\mathcal{V}''$ is any irreducible closed subset of $\mathcal{V}$ that contains $\mathcal{V}'$, we can apply \cite[Proposition 15.4~(1)]{GoertzWedhorn} to deduce that $\mathcal{V}''$ is either equal to $\mathcal{V}'$ or dominates $\Spec \mathcal{O}_K$. However, if $\mathcal{V}''$ dominates $\Spec \mathcal{O}_K$ and contains $\mathcal{V}'$, then applying the first assertion to $\mathcal{V}''$ shows that $\dim \mathcal{V}'' = \dim \mathcal{V}' + 1 = \dim \mathcal{V}$ and so $\mathcal{V}'' = \mathcal{V}$. It follows that $\codim_{\mathcal{V}}{\mathcal{V}'} = 1$ and so $\codim_{\mathcal{V}'}{\mathcal{W}} + 1 = \codim_{\mathcal{V}'}{\mathcal{W}} + \codim_{\mathcal{V}}{\mathcal{V}'}$. Since there is an order-reversing bijection between the closed irreducible subsets of $\mathcal{V}$ that contain $\mathcal{W}$ and the prime ideals of $\mathcal{O}_{\mathcal{V},\eta_\mathcal{W}}$ and since therefore every ascending chain of closed irreducible subsets of $\mathcal{V}$ that starts with $\mathcal{W}$, ends with $\mathcal{V}$, and does not admit a refinement has the same length by \cite[Proposition 14.105]{GoertzWedhorn}, it follows that
\[ \codim_{\mathcal{V}'}{\mathcal{W}} + \codim_{\mathcal{V}}{\mathcal{V}'} = \codim_{\mathcal{V}}{\mathcal{W}},\]
which concludes the analysis of the first case.

The second case occurs if $\mathcal{W}$ dominates $\Spec \mathcal{O}_K$. We have $\mathcal{O}_{\mathcal{V},\eta_{\mathcal{W}}} \simeq \mathcal{O}_{\mathcal{V}_{\eta},\eta_{\mathcal{W}}}$ by \cite[Proposition 4.20 and Remark 4.21]{GoertzWedhorn}, applied with $X = S = \Spec \mathcal{O}_K$, $Y = Z = \mathcal{V}$, $X' = \Spec K$, $f: X' \to X$ the morphism associated to $\eta$, and $z' = \eta_{\mathcal{W}}$. By \cite[(5.8.1)]{GoertzWedhorn} and the fact that $\mathcal{W}_\eta$ is irreducible since $\mathcal{W}$ is irreducible, this implies that
\[\codim_{\mathcal{V}}{\mathcal{W}} = \dim \mathcal{O}_{\mathcal{V},\eta_{\mathcal{W}}} = \dim \mathcal{O}_{\mathcal{V}_{\eta},\eta_{\mathcal{W}}} = \codim_{\mathcal{V}_\eta}{\mathcal{W}_\eta}. \]

Since $\mathcal{V}$ is irreducible, so is $\mathcal{V}_{\eta}$ and it then follows from the first assertion applied to $\pi$ as well as to $\pi|_{\mathcal{W}}: \mathcal{W} \to \Spec \mathcal{O}_K$ together with \cite[Proposition 5.30]{GoertzWedhorn} that
\[ \codim_{\mathcal{V}_\eta}{\mathcal{W}_\eta} = \dim \mathcal{V}_{\eta} - \dim \mathcal{W}_{\eta} = (\dim \mathcal{V}-1) - (\dim \mathcal{W}-1) = \dim \mathcal{V} - \dim \mathcal{W}\]
and we are done.
\end{proof}

We refer to \cite[Section 4.15]{GoertzWedhorn} for the definition of a group scheme $\mathcal{G} \to S$ over a base scheme $S$ and of homomorphisms between group schemes over $S$. In this article, ``subgroup scheme" always means ``closed subgroup scheme" in the sense of \cite[Definition 27.3]{GoertzWedhorn2}. For every base scheme $S$, we define the multiplicative group $\mathbb{G}_{m,S} := \Spec \mathbb{Z}[X^{\pm 1}] \times_{\Spec \mathbb{Z}} S$ over $S$, which is a commutative group scheme over $S$. If $S = \Spec R$ for a commutative ring $R$, we will also write $\mathbb{G}_{m,R}$ for $\mathbb{G}_{m,S}$. An abelian scheme over $S$ is a smooth proper group scheme over $S$ with geometrically connected fibers; all abelian schemes are commutative, see \cite[Corollary 27.103]{GoertzWedhorn2}. If $\mathcal{G}$ is a commutative group scheme, then we denote its ring of endomorphisms by $\End(\mathcal{G})$ and we denote the multiplication-by-$N$ endomorphism by $[N] = [N]_{\mathcal{G}}$ for $N \in \mathbb{Z}$.

A group scheme of finite type over the spectrum of a field $F$ is called an algebraic group over $F$. If $\mathcal{G}$ is a commutative algebraic group over a field $F$, then we use the convention that the degree $\deg [0]$ is $0$.

We conclude with the following definition that will be used in Conjectures \ref{conj:LIB} and \ref{conj:mainconj} to measure the ``size" of intersections between certain subvarieties of $\mathbb{G}^n_{m,S}$ for certain base schemes $S$.

\begin{defn}\label{defn:multiplicity}
Let $X$ be a Noetherian scheme, let $Z$ be an irreducible component of $X$, and let $\mathcal{O}_Z$ denote the local ring of $X$ at the generic point of $Z$. The multiplicity $\mult(Z)$ of $X$ at $Z$ is defined to be the length of $\mathcal{O}_Z$ as a module over itself (which is finite by \cite[Proposition B.36]{GoertzWedhorn}). 
\end{defn}

Note that, if $X$ is reduced, then $\mult(Z) = 1$ for every irreducible component $Z$ of $X$.

\section{Complexity of endomorphisms and of subgroup schemes of $\mathbb{G}_m^n$}\label{sec:complexity}

Let $B$ be either a non-empty open subset of $\Spec \mathcal{O}_K$ for a number field $K$ or a smooth irreducible curve over an algebraically closed field $F$ of characteristic $0$. In the latter case, we set $K$ equal to the function field of $B$. We fix an algebraic closure $\overline{K}$ of $K$ and we let $\eta$ denote the generic point of $B$. In this section, we define the complexity of endomorphisms of $\mathbb{G}^n_{m,K}$ and of certain closed subgroup schemes of $\mathbb{G}^n_{m,B} \to B$, where $n \in \mathbb{N}$.

We will use the following well-known lemma.

\begin{lem}\label{lem:endomorphisms}
There is a canonical isomorphism between the endomorphism ring of $\mathbb{G}^n_{m,K}$ and $\Mat_{n \times n}(\mathbb{Z})$ defined by
\[
\left[(X_1,\hdots,X_n) \mapsto (\prod_{j=1}^{n}{X_{j}^{a_{1j}}},\hdots,\prod_{j=1}^{n}{X_{j}^{a_{nj}}}) \right] \mapsto \begin{pmatrix}a_{11} & \hdots & a_{1n} \\ \vdots & & \vdots \\ a_{n1} & \hdots & a_{nn}\end{pmatrix}.
\]
\end{lem}

\begin{proof}
This follows from \cite[Proposition 3.2.17]{BombieriGubler}.
\end{proof}

In the rest of this article, we will often use the canonical isomorphism from Lemma \ref{lem:endomorphisms} to identify endomorphisms of $\mathbb{G}^n_{m,K}$ with $n \times n$-matrices with integer entries. 

\begin{defn}
The \emph{complexity} of $\phi \in \End(\mathbb{G}^n_{m,K})$, corresponding to a matrix $A \in \Mat_{n \times n}(\mathbb{Z})$, is
\[ \mathcal{C}(\phi) = \lVert A \rVert_\infty.\]
\end{defn}

We now want to use the above to define the complexity of certain closed subgroup schemes of $\mathbb{G}^n_{m,B}$, more precisely, of those subgroup schemes that arise as the intersection of the kernels of finitely many endomorphisms of $\mathbb{G}^n_{m,B}$.

\begin{defn}\label{defn:complexity}
Let $\Phi_1,\hdots,\Phi_s \in \End(\mathbb{G}^n_{m,B})$ and set $\mathcal{H} = \ker(\Phi_1,\hdots,\Phi_s)\subseteq \mathbb{G}^n_{m,B}$. We define the \emph{complexity} of $\mathcal{H}$ as
\[ \mathcal{C}(\mathcal{H}) = \min_{\widetilde{\Phi}_1,\hdots,\widetilde{\Phi}_{\widetilde{s}}}\max_{i}{\mathcal{C}((\widetilde{\Phi}_i)_{K})},\]
where the minimum is taken over all finite sets of endomorphisms $\{\widetilde{\Phi}_1,\hdots,\widetilde{\Phi}_{\widetilde{s}}\}$ of $\mathbb{G}^n_{m,B}$ such that $\ker(\widetilde{\Phi}_1,\hdots,\widetilde{\Phi}_{\widetilde{s}}) = \mathcal{H}$.
\end{defn}

Note that, with these definitions, there are at most finitely many endomorphisms of $\mathbb{G}^n_{m,K}$ of bounded complexity and at most finitely many subgroup schemes of $\mathbb{G}^n_{m,B}$ of bounded complexity.

The subgroup schemes whose complexity is defined are precisely those that are flat over $B$ as we will show now.

\begin{prop}\label{prop:complexityflat}
Let $\mathcal{H} \subseteq \mathbb{G}^n_{m,B}$ be a closed subgroup scheme. Then the complexity of $\mathcal{H}$ is defined if and only if the morphism $\mathcal{H} \to B$ is flat. Furthermore, if the complexity of $\mathcal{H}$ is defined, then we have that $\dim \mathcal{H}_{b} = \dim \mathcal{H} - 1$ for every point $b \in B$.
\end{prop}

\begin{proof}
Suppose that the complexity of $\mathcal{H}$ is defined, so
\[ \mathcal{H} = \ker \Phi,\]
where $\Phi: \mathbb{G}^n_{m,B} \to \mathbb{G}^{sn}_{m,B}$ is a homomorphism, $s$ is a natural number, and the exponents denote fibered products over $B$.

By Lemma \ref{lem:endomorphisms}, the homomorphism $\Phi$ corresponds to $sn$ elements of $\mathbb{Z}^n$. We choose a basis of the lattice that they generate and consider the $n \times n$ matrix $A$ whose rows are the basis vectors and potentially some zero vectors. Let $\Phi': \mathbb{G}^n_{m,B} \to \mathbb{G}^n_{m,B}$ denote the endomorphism corresponding to $A$ under the isomorphism in Lemma \ref{lem:endomorphisms}. Then we have that $\ker \Phi = \ker \Phi'$. Furthermore, it follows from the elementary divisor theorem that there are automorphisms $\Omega$, $\widetilde{\Omega} $ of $\mathbb{G}^n_{m,B}$ such that $\widetilde{\Omega} \circ \Phi' \circ \Omega$ is given by $(X_1,\hdots,X_n) \mapsto (X_1^{d_1},\hdots, X_k^{d_k},1,\hdots,1)$ for $k \in \{0,\hdots,n\}$ and $d_1,\hdots,d_k \in \mathbb{N}$ with $d_1 \mid \cdots \mid d_k$. We deduce the following series of isomorphisms of $B$-schemes:
\[\ker \Phi = \ker \Phi' \simeq \ker(\widetilde{\Omega} \circ \Phi' \circ \Omega) \simeq B \times_{\Spec \mathbb{Z}} \Spec \mathbb{Z}[X_1^{\pm 1},\hdots,X_n^{\pm1}]/(X_1^{d_1}-1,\hdots,X_k^{d_k}-1).\]
The ring $\mathbb{Z}[X_1^{\pm 1},\hdots,X_n^{\pm1}]/(X_1^{d_1}-1,\hdots,X_k^{d_k}-1)$ is a free and therefore flat $\mathbb{Z}$-module. Since flatness is preserved under base change, also $\mathcal{H}$ is flat over $B$.

By flatness, every irreducible component $\widetilde{\mathcal{H}}$ of $\mathcal{H}$ dominates $B$ and so Lemma \ref{lem:dimandcodim} (if $B \subseteq \Spec \mathcal{O}_K$) and \cite[Theorem 14.116 and Remark 14.117]{GoertzWedhorn} (if $B$ is a curve over $F$) imply that $\dim \widetilde{\mathcal{H}}_{b} = \dim \widetilde{\mathcal{H}} - 1 = \dim \widetilde{\mathcal{H}}_{\eta}$ for every point $b \in B$ such that $\widetilde{\mathcal{H}}_b$ is non-empty. But $\mathcal{H}_\eta$ is equidimensional of dimension $n-k$ and the proposition follows since $\mathcal{H} \to B$ is surjective.

Suppose on the other hand that $\mathcal{H} \to B$ is flat. By \cite[Corollary 3.2.15]{BombieriGubler}, there exists an endomorphism $\phi$ of $\mathbb{G}_{m,K}^n$ such that $\mathcal{H}_\eta = \ker \phi$ and we can find an endomorphism $\Phi$ of $\mathbb{G}_{m,B}^n$ such that $\phi = \Phi_K$. So $\mathcal{H}$ and $\ker \Phi$ have identical generic fibers, but then they must be equal since, by flatness, each of them is the schematic image of its generic fiber, see \cite[Proposition 14.14]{GoertzWedhorn}.
\end{proof}

Let $\mathcal{H}$ be a subgroup scheme whose complexity is defined. In the proof of Proposition \ref{prop:complexityflat}, we showed that $\mathcal{H}$ is always the kernel of one single endomorphism of $\mathbb{G}^n_{m,B}$. We now show that this endomorphism can be chosen such that its complexity is bounded linearly in the complexity of $\mathcal{H}$.

\begin{lem}\label{lem:onlyonegenerator}
Let $\phi_1, \hdots, \phi_s$ be endomorphisms of $\mathbb{G}^n_{m,B}$, induced by matrices $M_1,\hdots,M_s \in \Mat_{n \times n}(\mathbb{Z})$ (as in Lemma \ref{lem:endomorphisms}), and let $\mathcal{H} = \ker(\phi_1,\hdots,\phi_s)$. Then there exists an endomorphism $\phi$ of $\mathbb{G}^n_{m,B}$, induced by a matrix $M \in \Mat_{n \times n}(\mathbb{Z})$, such that $\mathcal{H} = \ker \phi$ and $\lVert M \rVert_{\infty} \leq \max\{2^{n-2},1\}\max_{i}{\lVert M_i \rVert_{\infty}}$.
\end{lem}

\begin{proof}
Let $\Lambda \subseteq \mathbb{Z}^n$ be the free abelian group generated by the rows of all the $M_i$ and let $r \leq n$ denote its rank. Applying \cite[p. 13, Corollary 2]{CasselsBook} to a maximal linearly independent subset $\{a_1, \hdots, a_r\}$ of the rows of all the $M_i$ yields a basis $b_1, \hdots, b_r$ of $\Lambda$ such that $\lVert b_1 \rVert_\infty = \frac{1}{v_{11}}\lVert a_1 \rVert_\infty \leq \lVert a_1 \rVert_\infty$ and
\begin{align*} \lVert b_i \rVert_\infty &\leq \frac{1}{v_{ii}}(\lVert a_i \rVert_\infty + \sum_{j=1}^{i-1} (v_{ii}-1)\lVert b_j\rVert_\infty)\\
&\leq \max\{\lVert a_1 \rVert_\infty,\lVert a_i \rVert_\infty\} + \sum_{j=2}^{i-1}\lVert b_j\rVert_\infty \\
&\leq \max\{2^{i-2},1\}\max_{j \leq i}{\lVert a_j \rVert_\infty} \quad (i = 2, \hdots, r),\end{align*}
where $v_{ii} \in \mathbb{N}$ is as in \cite{CasselsBook} and the last inequality is obtained from the penultimate one by induction on $i$.

Let $M$ be the $n \times n$-matrix whose first $r$ rows are the vectors $b_1, \hdots, b_r$ while the following $n-r$ rows are all zero. Let $\phi$ be the endomorphism of $\mathbb{G}^n_{m,B}$ induced by $M$. Then $\ker(\phi_1,\hdots,\phi_s) = \ker \phi$ and $\lVert M \rVert_{\infty} \leq \max\{2^{r-2},1\}\max_{i}{\lVert M_i \rVert_{\infty}}$.
\end{proof}

Next, for two endomorphisms $\phi$ and $\psi$ of $\mathbb{G}^n_{m,B}$, we characterize when $\ker \phi = \ker \psi$ as closed subschemes of $\mathbb{G}^n_{m,B}$.

\begin{lem}\label{lem:comparetwogenerators}
Let $\phi$ and $\psi$ be endomorphisms of $\mathbb{G}^n_{m,B}$. Then $\ker \phi = \ker \psi$ if and only if there exists an automorphism $\chi$ of $\mathbb{G}^n_{m,B}$ such that $\chi \circ \phi = \psi$.
\end{lem}

\begin{proof}
If $\ker \phi = \ker \psi$, then, in particular, the kernels of the corresponding endomorphisms of $\mathbb{G}^n_{m,K}$, which we also denote by $\phi$ and $\psi$ by abuse of notation, coincide.

Let $M$ and $M'$ denote the matrices associated to $\phi$ and $\psi$ respectively and let $r_M$ and $r_{M'}$ denote their respective ranks. Using \cite[Proposition 3.2.10 and Corollary 3.2.15]{BombieriGubler} and the connectedness of $\phi(\mathbb{G}^n_{m,K})$, we can find an automorphism $\alpha$ of $\mathbb{G}^n_{m,K}$ such that $(\alpha \circ \phi)(\mathbb{G}^n_{m,K}) = \mathbb{G}^{r_M}_{m,K} \times_K \{1\}^{n-r_M}$. We can similarly find an automorphism $\beta$ of $\mathbb{G}^n_{m,K}$ such that $(\beta \circ \psi)(\mathbb{G}^n_{m,K}) = \mathbb{G}^{r_{M'}}_{m,K} \times_K \{1\}^{n-r_{M'}}$. Since $\phi$ and $\psi$ have the same kernel, we must have $r_M = r_{M'}$.

The endomorphisms $\alpha \circ \phi$ and $\beta \circ \psi$ induce surjective homomorphisms $\widetilde{\phi}, \widetilde{\psi}: \mathbb{G}^n_{m,K} \to \mathbb{G}^{r_M}_{m,K}$, which both have the same kernel by our assumption. By \cite[Theorem 5.13]{MilneAG}, there is an automorphism $\widetilde{\gamma}$ of $\mathbb{G}^{r_M}_{m,K}$ such that $\widetilde{\gamma} \circ \widetilde{\phi} = \widetilde{\psi}$. Setting $\widetilde{\gamma} = \gamma \times_K \id_{\mathbb{G}^{n-r_M}_{m,K}}$, we have that $\gamma \circ \alpha \circ \phi = \beta \circ \psi$.

All in all, we have shown that there exists an automorphism $\chi = \beta^{-1} \circ \gamma \circ \alpha$ of $\mathbb{G}^n_{m,K}$ such that $\chi \circ \phi = \psi$. Any such $\chi$ extends to an automorphism of $\mathbb{G}^n_{m,B}$ with the desired property. On the other hand, the existence of such a $\chi$ clearly implies that $\ker \phi = \ker \psi$.
\end{proof}

Using the previous results, we can now relate the complexity of a subgroup scheme, defined by one single endomorphism, with the mass of the associated matrix as introduced in Section \ref{subsec:generalities}. The assumption that the subgroup scheme is defined by one endomorphism is not a big restriction because of Lemma \ref{lem:onlyonegenerator}.

\begin{cor}\label{cor:complexityingm}
Let $\mathcal{H} \subseteq \mathbb{G}^n_{m,B}$ be a flat subgroup scheme. Let $\phi$ and $\psi$ be endomorphisms of $\mathbb{G}^n_{m,B}$, induced by matrices $M, A \in \Mat_{n \times n}(\mathbb{Z})$ respectively, such that $\mathcal{H} = \ker \phi = \ker \psi$. Then $m(M) = m(A)$ and
  \[\max\{2^{n-2},1\}^{-1}m(M) \leq \mathcal{C}(\mathcal{H}) \leq m(M).
\]
\end{cor}

\begin{proof}
We have $m(M) = m(A)$ by Lemma \ref{lem:comparetwogenerators}. The upper bound for $\mathcal{C}(\mathcal{H})$ follows directly from the definition of the complexity while the lower bound follows from Lemma \ref{lem:onlyonegenerator} and the first part of the corollary.
\end{proof}

The arithmetic unlikely intersections that we will consider later have been studied most extensively in the case where one intersects with the kernel of multiplication by $N$ for some $N \in \mathbb{Z}$. It is therefore worthwhile to compute the complexity of this kernel and we do so now.

\begin{lem}\label{lem:complexitymultbyN}
Let $n \in \mathbb{N}$. Let $\mathcal{H} \subseteq \mathbb{G}^n_{m,B}$ denote the kernel of multiplication by $N \in \mathbb{Z}$. Then $\mathcal{C}(\mathcal{H}) = |N|$.
\end{lem}

\begin{proof}
It is clear that $\mathcal{C}(\mathcal{H}) \leq |N|$ since $\mathcal{H}$ is defined by the monomial equations $X_1^N = \cdots = X_n^N = 1$. If $N = 0$, this already proves the lemma.

To show the inequality in the other direction if $N \neq 0$, suppose that $\phi_1, \hdots, \phi_m$ are endomorphisms of $\mathbb{G}^n_{m,B}$ such that $\mathcal{H} = \ker(\phi_1,\hdots,\phi_m)$. If we denote the corresponding endomorphisms of $\mathbb{G}^n_{m,K}$ also by $\phi_1, \hdots, \phi_m$ and if $[N]$ denotes the multiplication by $N$ on $\mathbb{G}^n_{m,K}$, we get that $\ker [N] \subseteq \ker \phi_i$ for all $i$. By \cite[Theorem 5.13]{MilneAG}, there are endomorphisms $\chi_i$ of $\mathbb{G}^n_{m,K}$ such that $\phi_i = \chi_i \circ [N]$ ($i = 1,\hdots,n$).

Since $N \neq 0$, there exists some $i$ such that $\phi_i \neq 0$ and so $\chi_i \neq 0$. Then $\mathcal{C}(\phi_i) = |N|\mathcal{C}(\chi_i) \geq |N|$ and we are done.
\end{proof}

\section{Bounding the size of arithmetic likely intersections}\label{sec:trivialbound}

Let $K$ be a number field. In this section, we study the ``size" of a possibly \emph{likely} intersection between a fixed integral closed subscheme of $\mathbb{G}^n_{m,\mathcal{O}_K}$ that dominates $\Spec \mathcal{O}_K$ and a varying flat subgroup scheme. The aim is to put this ``size" into relation with the complexity of the subgroup scheme that we are intersecting with as defined in Section \ref{sec:complexity}. This is partially achieved in Theorem \ref{thm:LIB} and gives us some idea of what we should expect in the case of an \emph{unlikely} intersection.

For example, the intersection between the Zariski closure in $\mathbb{G}_{m,\mathbb{Z}}$ of $\{a\} \subseteq \mathbb{G}_{m,\mathbb{Q}}(\mathbb{Q})$, $a \in \mathbb{Z}\backslash\{0\}$, and the kernel of multiplication by $N \in \mathbb{N}$ in $\mathbb{G}_{m,\mathbb{Z}}$ is isomorphic to $\Spec \mathbb{Z}/(a^N-1)\mathbb{Z}$. If $a \neq \pm 1$, then the ring $\mathbb{Z}/(a^N-1)\mathbb{Z}$ is finite for all $N \in \mathbb{N}$ and the logarithm of its cardinality, which measures the ``size" of the intersection, is bounded from above by $(2\log|a|)N$. This is the product of $N$, which is equal to the complexity of the kernel of multiplication by $N$ by Lemma \ref{lem:complexitymultbyN}, and a constant depending on $a$, but not on $N$.

We next consider the more general setting of the intersection of a fixed integral closed subscheme $\mathcal{V} \subseteq \mathbb{G}^n_{m,\mathcal{O}_K}$ that dominates $\Spec \mathcal{O}_K$ with a varying flat subgroup scheme $\mathcal{H} \subseteq \mathbb{G}^n_{m,\mathcal{O}_K}$ such that $\dim \mathcal{V} + \dim \mathcal{H} \leq n+1$. Inspired by the example above, we shall look for an upper bound of the form $C\mathcal{C}(\mathcal{H})^{\dim \mathcal{V}}$ for the ``size" of this intersection. Here, the ``size" is the sum of the degrees (with respect to some fixed projective immersion) of all irreducible components of $\mathcal{V} \cap \mathcal{H}$ that do \emph{not} dominate $\Spec \mathcal{O}_K$, weighted by the multiplicity of the component and the size of the residue field.

Since the intersection $\mathcal{V} \cap \mathcal{H}$ is allowed to be likely here, we refer to this bound, which we state in full generality in Conjecture \ref{conj:LIB}, as the \textit{Likely Intersection Bound} (LIB). If $\mathcal{V} \cap \mathcal{H}$ is finite, then its ``size" is just the logarithm of the cardinality of $\mathcal{O}_{\mathcal{V} \cap \mathcal{H}}(\mathcal{V} \cap \mathcal{H})$ (see Remark \ref{rmk:lhszerodim} below), so in particular Conjecture \ref{conj:LIB} would yield the bound appearing in the example above with some unspecified constant $C$ in place of $2\log|a|$.

We do not know how to prove LIB in general. However, in Theorem \ref{thm:LIB}, we prove LIB ``without multiplicities", which in the example above amounts to replacing $a^N-1$ by the product of all primes that divide it. A more complicated example can be found in Section \ref{subsec:trivialboundex}. Furthermore, the methods we use to prove Theorem \ref{thm:LIB} allow us to improve one of the main results of
\cite{BCMOS_Preprint}.

The rest of this section is structured as follows: in Section \ref{subsec:trivialboundex}, we present an example in $\mathbb{G}^2_{m,\mathbb{Z}}$, where LIB (with multiplicities) can be proved directly. In Section \ref{subsec:prooftrivialboundnomult}, we state Theorem \ref{thm:LIB} and the general form of LIB as Conjecture \ref{conj:LIB}. Then, in Section \ref{subsec:arithmeticintersectiontheory}, we recall several important definitions and results from arithmetic intersection theory that we will use. Finally, in Section \ref{subsec:proof_LIB}, we prove Theorem \ref{thm:LIB}.

In the case of certain unlikely intersections, one can prove a better upper bound than LIB, which we call the \textit{Unlikely Intersection Bound} (ULIB). This will be the object of Section \ref{sec:mainconj}.

\subsection{Introducing the Likely Intersection Bound: an example}\label{subsec:trivialboundex}

Let $P(T)$ be a polynomial in $\mathbb{Z}[T]$, not equal to $\pm T^M$ for any $M \in \mathbb{N} \cup \{0\}$. We consider the closed subscheme $\mathcal{V} \subseteq \mathbb{G}^2_{m,\mathbb{Z}} = \Spec \mathbb{Z}[X,X^{-1},Y,Y^{-1}]$ defined by $Y = P(X)$, which we will intersect with the kernel of multiplication by $N \in \mathbb{N}$. Since $\dim \mathcal{V} = 2$ and the kernel of multiplication by $N$ has codimension $2$ in $\mathbb{G}^2_{m,\mathbb{Z}}$, this is a likely intersection. We have:
\begin{align*}
\mathcal{V} \cap \ker [N] &= \Spec \mathbb{Z}[X,X^{-1},Y,Y^{-1}]/(X^N-1,Y^N-1,Y-P(X)) \\
&\simeq \Spec \mathbb{Z}[T]/(T^N-1,P(T)^N-1).
\end{align*}

This intersection can have irreducible components that dominate $\Spec \mathbb{Z}$. While we will be able to control positive-dimensional irreducible components contained in some fiber of positive characteristic via their degree, we will always exclude irreducible components that dominate the base from our considerations. To get rid of them, we set 
\[ Q_N(T) = \prod_{\stackrel{d \mid N}{P(\zeta_d)^N = 1}}{\Psi_d(T)},\]
where $\zeta_d \in \mathbb{C}$ denotes a primitive $d$-th root of unity and we recall that $\Psi_d$ denotes the $d$-th cyclotomic polynomial. The intersection without the irreducible components that dominate $\Spec \mathbb{Z}$ is then isomorphic to
\[ \Spec \mathbb{Z}[T,Q_N(T)^{-1}]/(T^N-1,P(T)^N-1).\]
The following proposition therefore establishes LIB with multiplicities for $\mathcal{V} \cap \ker [N]$.

\begin{prop}\label{prop:LIBexample}
    There exists a constant $c > 0$ such that
    \begin{equation}\label{eq:likely_bound_example}
    \log |\mathbb{Z}[T,Q_N(T)^{-1}]/(T^N-1,P(T)^N-1)| \leq cN^2
    \end{equation}
    for all $N \in \mathbb{N}$.
\end{prop}

\begin{proof}
Let $N \in \mathbb{N}$ and set $Q(T) = Q_N(T)$. Moreover, we set $\Sigma = \Sigma_N = \{d \in \mathbb{N};\,d \mid N, Q(\zeta_d) \neq 0\}$.

We define a homomorphism of abelian groups
\[\rho: \prod_{d \in \Sigma}{\mathbb{Z}[T,Q(T)^{-1}]/(\Psi_d(T),P(T)^N-1)} \to \mathbb{Z}[T,Q(T)^{-1}]/(T^N-1,P(T)^N-1)\]
by
\[
    \rho:(A_d)_{d \in \Sigma} \mapsto \sum_{d \in \Sigma}{A_d\prod_{\stackrel{j \mid N}{j \neq d}}{\Psi_j}}.
\]
To prove the proposition, we bound first the size of the domain of $\rho$ and then the index of its image inside the codomain.

The domain of $\rho$ is isomorphic to
\[\prod_{d \in \Sigma}{\mathbb{Z}[\zeta_d,Q(\zeta_d)^{-1}]/(P(\zeta_d)^N-1)}.\]
 Note that $\mathbb{Z}[\zeta_d,Q(\zeta_d)^{-1}]$ is equal to the ring of $S$-integers inside $\mathbb{Q}(\zeta_d)$, where $S$ is the finite set of all finite places of $\mathbb{Q}(\zeta_d)$ that divide $Q(\zeta_d)$. The maximal ideals in the ring of $S$-integers are in a canonical bijection with the maximal ideals in the ring of integers that do not belong to $S$ and this bijection preserves the ideal norm. The size of the domain of $\rho$ is therefore bounded from above by
\[\prod_{d \in \Sigma}{|\mathrm{Nm}_{\mathbb{Q}(\zeta_d)/\mathbb{Q}}(P(\zeta_d)^N-1)|} \leq \prod_{d \mid N}{H(P(\zeta_d)^N-1)^{\phi(d)}} \leq C^{N\sum_{d \mid N}{\phi(d)}} = C^{N^2},\]
where $H$ denotes the absolute multiplicative Weil height and $C$ is a constant that depends only on $P$.

In order to bound the index of the image of $\rho$, we consider the sequence of homomorphisms of abelian groups
\[
\prod_{d \in \Sigma} \mathbb{Z}[T,Q(T)^{-1}]/(\Psi_d(T)) \xrightarrow{\widetilde{\rho}} \mathbb{Z}[T,Q(T)^{-1}]/(T^N-1) \xhookrightarrow{\sigma} \prod_{d\in \Sigma} \mathbb{Z}[T,Q(T)^{-1}]/(\Psi_d(T)),
\]
where
\[
    \widetilde{\rho}:(A_d)_{d \in \Sigma} \mapsto \sum_{d \in \Sigma}{A_d\prod_{\stackrel{j \mid N}{j \neq d}}{\Psi_j}}
\]
induces $\rho$ and $\sigma$ is induced by the natural projections. Note that $\sigma$ is injective and that the index of the image of $\rho$ inside its codomain divides the index of the image of $\widetilde{\rho}$ inside its codomain.

The composition $\sigma \circ \widetilde{\rho}$ corresponds to multiplication by $\prod_{\stackrel{j \mid N}{j \neq d}}{\Psi_{j}(\zeta_d)}$ on the factor $\mathbb{Z}[\zeta_d,Q(\zeta_d)^{-1}]$. Thus, the index of the image of $\sigma \circ \widetilde{\rho}$ inside its codomain divides
\[ \textstyle{\prod_{d \in \Sigma}\prod_{\stackrel{j \mid N}{j \neq d}}{|\mathrm{Nm}_{\mathbb{Q}(\zeta_d)/\mathbb{Q}}(\Psi_{j}(\zeta_d))|}}, \]
which in turn divides $\prod_{i=1}^{N}\prod_{\stackrel{j=1}{j \neq i}}^{N}{|\zeta_N^i - \zeta_N^j|} = |\mathrm{disc}(T^N-1)| = N^N$.

As $\sigma$ is injective, we deduce that the index of the image of $\widetilde{\rho}$ inside its codomain also divides $N^{N}$. Hence, the same holds for the index of the image of $\rho$ inside its codomain. Putting everything together, we find that the cardinality of the ring $\mathbb{Z}[T,Q(T)^{-1}]/(T^N-1,P(T)^N-1)$ is bounded from above by $N^NC^{N^2}$ and we are done.

\end{proof}

\begin{rmk}
Let $\mathcal{D}$ denote the set of $d \in \mathbb{N}$ such that the $d$-th cyclotomic polynomial $\Psi_d$ divides $P(T)^N-1$ for some $N \in \mathbb{N}$. As $P(T) \neq \pm T^M$, the curve parametrized by $(T,P(T))$ is not a component of an algebraic subgroup of $\mathbb{G}^2_{m,\mathbb{Q}}$. Therefore, the set $\mathcal{D}$ is finite by the theorem of Ihara-Serre-Tate \cite{Lang_1965}. The polynomial $Q_N(T)$ always divides $\prod_{d \in \mathcal{D}}{\Psi_d(T)}$, no matter what $N$ is.
\end{rmk}

In Proposition \ref{prop:LIBexample}, $N$ is the complexity of $\ker [N]$ by Lemma \ref{lem:complexitymultbyN}. Note that the exponent $2$ in the right-hand side of \eqref{eq:likely_bound_example} coincides with $\dim \mathcal{V}$.

We will now show that the order of magnitude of the bound \eqref{eq:likely_bound_example} cannot be improved if $P(T) = T+1$. In this special case, we can even give an asymptotic for the left-hand side of \eqref{eq:likely_bound_example} as long as $6 \nmid N$, which implies that $Q_N(T) = 1$ (intersect two circles with radius $1$ around $0$ and $1$ in the complex plane).

\begin{prop} \label{prop:Lfunction}
There exists $\kappa > 0$ such that
\[
    \log \left |\mathbb{Z}[T]/(T^N-1,(T+1)^N-1) \right| = \frac{3\sqrt{3}}{4\pi}L(2,\chi_3) N^2 + \mathcal{O}(N^{2-\kappa}),
\]
where $N \in \mathbb{N}$ such that $6 \nmid N$ and $\chi_3$ denotes the non-trivial character modulo $3$.
\end{prop}

We have
\[\frac{3\sqrt{3}}{4\pi}L(2,\chi_3) = 0.3230659472194505140936\hdots > 0,\]
 see \cite{Habegger_2009}.

\begin{proof}
   Let $N \in \mathbb{N}$ such that $6 \nmid N$. This implies that $T^N-1$ and $(T+1)^N-1$ do not have any common zeroes in $\mathbb{C}$. It follows from \cite[Lemma 1]{Frenkel_Zabradi_2018} that
\begin{equation}\label{eq:atoral}
\log |\mathbb{Z}[T]/(T^N-1,(T+1)^N-1)| = \log |\mathrm{res}(T^N-1, (T+1)^N-1)| = \sum_{j = 0}^{N-1}\sum_{k=0}^{N-1}{\log |\zeta_N^j+1-\zeta_N^k|},
\end{equation}
where $\mathrm{res}$ denotes the resultant.

We now apply \cite[Theorem 1.2]{Dimitrov_Habegger_2019} with $G=\{(\zeta_N^i, \zeta_N^j);\, i,j \in \mathbb{Z}\}$ and $P(X,Y)=X-Y+1$, which is essentially atoral as defined in \cite[p. 2]{Dimitrov_Habegger_2019}. It follows that \eqref{eq:atoral} is equal to
\[ m(X-Y+1)N^2 + \mathcal{O}(N^{2-\kappa}) \]
for some $\kappa > 0$, where $m(X- Y+1)$ denotes the logarithmic Mahler measure of the bivariate polynomial $X-Y+1$. Note that
\[ m(X-Y+1) = m(X+Y-1) = \frac{3\sqrt{3}}{4\pi}L(2,\chi_3)\]
by \cite[Example 5]{Smyth_1981}, where $\chi_3$ is the non-trivial character modulo $3$. Putting everything together, we obtain
\[
    \log \left |\mathbb{Z}[T]/(T^N-1,(T+1)^N-1) \right| = \frac{3\sqrt{3}}{4\pi}L(2,\chi_3) N^2 + \mathcal{O}(N^{2-\kappa}),
\]
which proves the proposition.
\end{proof}

\subsection{The Likely Intersection Bound}\label{subsec:prooftrivialboundnomult}

We begin by introducing the setting in which we will work for the rest of this subsection. Let $K$ be a number field, let $\eta$ denote the generic point of $\Spec \mathcal{O}_K$, and let $n \in \mathbb{N}$. If $\mathcal{W} \subseteq \mathbb{G}^n_{m,\mathcal{O}_K}$ is an irreducible closed subscheme that does not dominate $\Spec \mathcal{O}_K$, we set $N(\mathcal{W}) = N(\mathfrak{P})$, where $\mathfrak{P}$ is the maximal ideal of $\mathcal{O}_K$ such that $\mathcal{W} \subseteq \mathbb{G}^n_{m,\mathbb{F}_{\mathfrak{P}}}$. If moreover $\mathcal{W}$ is reduced and $\overline{\mathcal{W}}$ denotes its schematic closure in $\mathbb{P}^{n}_{\mathcal{O}_K}$ (which is reduced), we set $\deg \mathcal{W} = \deg \overline{\mathcal{W}}$, where $\deg$ is the usual projective degree in $\mathbb{P}^{n}_{\mathbb{F}_{\mathfrak{P}}}$.

We can now state LIB ``without multiplicities'':

\begin{thm}\label{thm:LIB}
Let $\mathcal{V} \subseteq \mathbb{G}^n_{m,\mathcal{O}_K}$ be an integral closed subscheme that dominates $\Spec \mathcal{O}_K$. There exists a constant $C = C(n,K,\mathcal{V}) \in \mathbb{R}$ such that for all flat subgroup schemes $\mathcal{H} \subseteq \mathbb{G}^n_{m,\mathcal{O}_K}$ which satisfy $\dim \mathcal{H} + \dim \mathcal{V} \leq n+1$, we have
\begin{equation}\label{eq:trivialbound}
\sum_{\mathcal{Z}}{\deg(\mathcal{Z}_{\red})\log N(\mathcal{Z})} \leq C\mathcal{C}(\mathcal{H})^{\dim \mathcal{V}},
\end{equation}
where the sum on the left-hand side runs over the irreducible components of $\mathcal{V} \cap \mathcal{H}$ which do not dominate $\Spec \mathcal{O}_K$.
\end{thm}

Theorem \ref{thm:LIB} will be proved in Section \ref{subsec:proof_LIB}. We note that if either $\dim \mathcal{V} \leq 2$ and $\mathcal{V}_{\eta}$ is not contained in a proper algebraic subgroup of $G$ or $\dim \alpha(\mathcal{V}_\eta) = \min\{\dim \mathcal{V}_\eta,k\}$ for every surjective homomorphism of algebraic groups $\alpha: \mathbb{G}^n_{m,K} \to \mathbb{G}^k_{m,K}$, then thanks to \cite[Th\'eor\`eme 1.2]{Maurin08} and \cite[Th\'eor\`eme 1.1]{Maurin11}, the union of all irreducible components of $\mathcal{V} \cap \mathcal{H}$ that dominate $\Spec \mathcal{O}_K$, for varying $\mathcal{H}$ with $\dim \mathcal{H} + \dim \mathcal{V} \leq n+1$, is not Zariski dense in $\mathcal{V}$. 

The following remark shows how to connect Theorem \ref{thm:LIB} with the bound obtained in Section \ref{subsec:trivialboundex}.

\begin{rmk}\label{rmk:lhszerodim}

In the setting of Theorem \ref{thm:LIB}, the image of $\mathcal{V} \cap \mathcal{H} \hookrightarrow \mathbb{P}^n_{\mathcal{O}_K}$ will not be closed in general. However, if $\mathcal{Z}$ is a zero-dimensional irreducible component of $\mathcal{V} \cap \mathcal{H}$, then the image of $\mathcal{Z} \to \mathbb{P}^n_{\mathcal{O}_K}$ is a closed point. If moreover $\mathcal{V} \cap \mathcal{H}$ itself is zero-dimensional, which we will assume for the rest of this remark, then it follows from \cite[Proposition 5.11]{GoertzWedhorn} that $\mathcal{V} \cap \mathcal{H} \simeq \Spec R$, where $R = \prod_{\mathcal{Z}}{\mathcal{O}_{\mathcal{Z}}}$, the product runs over the (finitely many) irreducible components of $\mathcal{V} \cap \mathcal{H}$, and, for each irreducible component $\mathcal{Z} = \{z\}$, the ring $ \mathcal{O}_{\mathcal{Z}}$ is the local ring of $\mathcal{V} \cap \mathcal{H}$ at $z$. Then $\mathcal{O}_{\mathcal{Z}}$ is a local ring with finite residue field $k$. If $\mathfrak{P}$ denotes the maximal ideal to which $\mathcal{Z}$ maps in $\Spec \mathcal{O}_K$, then $k$ is a finite extension of $\mathbb{F}_{\mathfrak{P}}$. In this situation, we have $\log N(\mathcal{Z}) = \log |\mathbb{F}_{\mathfrak{P}}|$ and $\deg(\mathcal{Z}_{\mathrm{red}}) = [k:\mathbb{F}_{\mathfrak{P}}]$. We deduce that 
\begin{equation}\label{eq:logcardreduced}
    \sum_{\mathcal{Z}}{\deg(\mathcal{Z}_{\mathrm{\red}})\log N(\mathcal{Z})} = \log |R_{\red}|,
\end{equation}
where the left-hand side is equal to the left-hand side of \eqref{eq:trivialbound}.

In order to make a connection with the left-hand side of \eqref{eq:likely_bound_example}, we need a more refined equality: recall that, in Definition \ref{defn:multiplicity}, we have defined the multiplicity $\mult(\mathcal{Z})$ of $\mathcal{V} \cap \mathcal{H}$ at $\mathcal{Z}$ to be the length of $\mathcal{O}_{\mathcal{Z}}$ as a module over itself. Since $\mathcal{O}_{\mathcal{Z}}$ has a unique maximal ideal, any simple $\mathcal{O}_{\mathcal{Z}}$-module is isomorphic to $k$ (with its canonical structure as an $\mathcal{O}_{\mathcal{Z}}$-module) and thus $\mult(\mathcal{Z}) = (\log |\mathcal{O}_{\mathcal{Z}}|)/(\log |k|)$ (see \cite[p.~557]{GoertzWedhorn}). Altogether, we obtain that $\mult(\mathcal{Z})\deg(\mathcal{Z}_{\mathrm{\red}})\log N(\mathcal{Z}) = \log |\mathcal{O}_{\mathcal{Z}}|$ and therefore
\begin{equation}\label{eq:logcard}
    \sum_{\mathcal{Z}}{\mult(\mathcal{Z})\deg(\mathcal{Z}_{\mathrm{\red}})\log N(\mathcal{Z})} = \log |R|.
\end{equation}
If $R$ is reduced, then $\mult(\mathcal{Z}) = 1$ for all $\mathcal{Z}$ and \eqref{eq:logcard} reduces to \eqref{eq:logcardreduced}.
\end{rmk}

Remark \ref{rmk:lhszerodim} motivates the following general conjectural statement of LIB:

\begin{conj}[Likely Intersection Bound] \label{conj:LIB}
Let $\mathcal{V} \subseteq \mathbb{G}^n_{m,\mathcal{O}_K}$ be an integral closed subscheme that dominates $\Spec \mathcal{O}_K$. There exists a constant $C = C(n,K,\mathcal{V}) \in \mathbb{R}$ such that for all flat subgroup schemes $\mathcal{H} \subseteq \mathbb{G}^n_{m,\mathcal{O}_K}$ which satisfy $\dim \mathcal{H} + \dim \mathcal{V} \leq n+1$, we have
\begin{equation}\label{eq:trivialbound2}
\sum_{\mathcal{Z}}{\mult(\mathcal{Z})\deg(\mathcal{Z}_{\red})\log N(\mathcal{Z})} \leq C\mathcal{C}(\mathcal{H})^{\dim \mathcal{V}},
\end{equation}
where the sum on the left-hand side runs over the irreducible components $\mathcal{Z}$ of $\mathcal{V} \cap \mathcal{H}$ which do not dominate $\Spec \mathcal{O}_K$.
\end{conj}

If $\dim \mathcal{V} = 1$, one can follow a similar reasoning as in the proof of Theorem \ref{thm:ULIB_dim1} below to directly prove Conjecture \ref{conj:LIB}, see Remark \ref{rmk:LIB_dim1}. Proposition \ref{prop:LIBexample} shows that Conjecture \ref{conj:LIB} also holds for the $2$-dimensional subscheme $\mathcal{V} \subseteq \mathbb{G}^2_{m,\mathbb{Z}} = \Spec \mathbb{Z}[X^{\pm 1},Y^{\pm 1}]$ defined by $Y = P(X)$ if we restrict $\mathcal{H}$ to be the kernel of raising to the $N$-th power for $N \in \mathbb{N}$. It seems plausible that one could deduce more instances of Conjecture \ref{conj:LIB} from \cite[Conjecture 2]{Silverman_2017} by adapting the argument in Section \ref{subsec:trivialboundex}.

\subsection{Arithmetic intersection theory}\label{subsec:arithmeticintersectiontheory}
Let $K$ be a number field with a fixed algebraic closure $\overline{K}$ and $n \in \mathbb{N}$. In Section \ref{subsec:proof_LIB}, we will use heights of integral closed subschemes of $\mathbb{P}^n_{\mathcal{O}_K}$ as defined in \cite{BGS} and we will apply the arithmetic B\'ezout theorem from \cite{BGS} to prove LIB without multiplicities for $\mathbb{G}^n_{m,\mathcal{O}_K}$. In this subsection, we recall,  for the reader's convenience, the definitions as well as the results in the form in which we will use them. First of all, consider the trivial hermitian vector bundle $E:=\mathcal{O}_{\Spec \mathcal{O}_K}^{n+1}$ on $\Spec \mathcal{O}_K$. Its projectivization $\mathbb{P}(E)$ is an $\mathcal{O}_K$-scheme that is canonically isomorphic to the projective space $\mathbb{P}_{\mathcal{O}_K}^n$.

Let $\mathcal{Z} \subseteq \mathbb{P}(E) \simeq \mathbb{P}^n_{\mathcal{O}_K}$ be an integral closed subscheme. We denote the height of $\mathcal{Z}$ as defined in \cite[Definition 4.1.1]{BGS} by $h(\mathcal{Z})$. If $\mathcal{Z} \subseteq \mathbb{P}^n_{\mathbb{F}_{\mathfrak{P}}}$ for some maximal ideal $\mathfrak{P}$ of $\mathcal{O}_K$, then $\deg \mathcal{Z}$ denotes the usual degree of $\mathcal{Z}$ as a projective variety. If $\mathcal{Z}$ dominates $\Spec \mathcal{O}_K$, then $\deg \mathcal{Z}$ denotes the degree of the generic fiber of $\mathcal{Z}$.

For every $p \geq 0$ the definitions of height and degree can be extended by linearity to any $p$-dimensional cycle $Z \in Z_p(\mathbb{P}^n_{\mathcal{O}_K})$. Equivalently, the height function $h$ can be defined directly for cycles as in \cite[Definition 4.1.1]{BGS} and then the height defined in this way is actually linear because of \cite[Proposition 2.3.1 (i)]{BGS} and the fact that the arithmetic degree $\widehat{\deg}: \widehat{\text{CH}}^1(\Spec \mathcal{O}_K) \to \mathbb{R}$ is a group homomorphism.

\begin{rmk}\label{rmk:heightinclosedfiber}
If $\mathcal{Z}$ is an integral closed subscheme of $\mathbb{P}^n_{\mathbb{F}_{\mathfrak{P}}}$ for some maximal ideal $\mathfrak{P}$ of $\mathcal{O}_K$, then $h(\mathcal{Z}) = (\deg \mathcal{Z})\log N(\mathfrak{P})$ by \cite[Example 3.1.2.2, Equation (3.1.5), and Proposition 4.1.2 (i)]{BGS}. Note that the correction term in \cite[Proposition 4.1.2 (i)]{BGS} vanishes if $Z$ is supported on a finite union of closed fibers.
\end{rmk}

\begin{rmk}\label{rmk:heightnonnegative}
If $\mathcal{Z} \subseteq \mathbb{P}^n_{\mathcal{O}_K}$ is any integral closed subscheme, then $h(\mathcal{Z}) \geq 0$ by \cite[Theorem 5.2.3]{BGS}.
\end{rmk}

In the next lemma, we recall some useful basic facts about integral closed subschemes of $\mathbb{P}^n_{\mathcal{O}_K}$.

\begin{lem}\label{lem:degreereduction}
Let $\mathcal{V} \subseteq \mathbb{P}^n_{\mathcal{O}_K}$ be an integral closed subscheme that dominates $\Spec \mathcal{O}_K$. Then the following hold:
\begin{enumerate}
    \item the induced morphism $\pi: \mathcal{V} \to \Spec \mathcal{O}_K$ is flat and surjective,
    \item for every (not necessarily closed) point $p \in \Spec \mathcal{O}_K$, the fiber $\pi^{-1}(p)$, endowed with the reduced scheme structure, is equidimensional of dimension $\dim \mathcal{V}-1$ and has degree at most $\deg \mathcal{V}$, and
    \item for every open $U \subseteq \mathbb{P}^n_{\mathcal{O}_K}$ such that $\mathcal{V} \cap U \neq \emptyset$, $\mathcal{V} \cap U$ is integral of dimension $\dim \mathcal{V}$.
\end{enumerate}
\end{lem}

\begin{proof}
The morphism $\mathcal{V} \to \Spec \mathcal{O}_K$ is flat thanks to \cite[Proposition 14.14]{GoertzWedhorn}. It is also proper, hence surjective. This proves (1). By Lemma \ref{lem:dimandcodim}~(1),
the fiber $\pi^{-1}(p)$ is equidimensional of dimension $\dim \mathcal{V}-1$. The bound for the degree then follows from \cite[Theorem III.9.9]{Hartshorne}, so (2) holds. For (3), $\mathcal{V} \cap U$ is integral since $\mathcal{V}$ is integral, and $\dim (\mathcal{V} \cap U) = \dim \mathcal{V}$ by Lemma \ref{lem:dimandcodim}.
\end{proof}

We are now ready to state the arithmetic B\'ezout theorem in the form in which we will use it.

\begin{thm}\label{thm:arithmeticbezout}
There exists a constant $c = c([K:\mathbb{Q}],n)$ such that the following holds: let $\mathcal{V}$ and $\mathcal{H}$ be integral closed subschemes of $\mathbb{P}^n_{\mathcal{O}_K}$. Suppose that $\dim \mathcal{H} = n$ and $\mathcal{V} \not\subseteq \mathcal{H}$. Then $\sum_{\mathcal{Z}}{h(\mathcal{Z})} \leq h(\mathcal{V})\deg \mathcal{H}+h(\mathcal{H})\deg \mathcal{V}+c\deg \mathcal{V}\deg \mathcal{H}$, where the sum runs over the irreducible components $\mathcal{Z}$ of $\mathcal{V} \cap \mathcal{H}$, endowed with the reduced scheme structure.
\end{thm}

\begin{proof}
Suppose first that $\mathcal{V} \subseteq \mathbb{P}^n_{\mathbb{F}_{\mathfrak{P}}}$ for some maximal ideal $\mathfrak{P}$ of $\mathcal{O}_K$ and $\mathcal{V} \cap \mathcal{H} \neq \emptyset$. Since $\mathcal{V} \not\subseteq \mathcal{H}$, we deduce that $\mathcal{H}$ dominates $\Spec \mathcal{O}_K$. By Lemma \ref{lem:degreereduction}~(2), $\mathcal{H} \cap \mathbb{P}^n_{\mathbb{F}_{\mathfrak{P}}}$, endowed with the reduced scheme structure, is equidimensional of dimension $n-1$ and of degree at most $\deg \mathcal{H}$. The theorem now follows from Remark \ref{rmk:heightinclosedfiber} and the classical B\'ezout theorem \cite[Example 8.4.6]{Fulton}.

We now treat the case where $\mathcal{V}$ dominates $\Spec \mathcal{O}_K$. If $\mathcal{H}$ does not dominate $\Spec \mathcal{O}_K$, then $\mathcal{H} = \mathbb{P}^n_{\mathbb{F}_{\mathfrak{P}}}$ for some maximal ideal $\mathfrak{P}$ of $\mathcal{O}_K$. By Lemma \ref{lem:degreereduction}~(2), $\mathcal{V} \cap \mathbb{P}^n_{\mathbb{F}_{\mathfrak{P}}}$, endowed with the reduced scheme structure, has degree at most $\deg \mathcal{V}$, so the theorem follows from Remark \ref{rmk:heightinclosedfiber}.

Hence, we can assume that both $\mathcal{V}$ and $\mathcal{H}$ dominate $\Spec \mathcal{O}_K$. We want to show that each irreducible component of $\mathcal{V} \cap \mathcal{H}$ has dimension $\dim \mathcal{V}-1$. This implies that $\mathcal{V}$ and $\mathcal{H}$ meet properly on $\mathbb{P}^n_K$ so that we can apply the arithmetic B\'ezout theorem \cite[Theorem 4.2.3]{BGS}. Let therefore $\mathcal{Z}$ be such an irreducible component and let $\eta$ denote the generic point of $\Spec \mathcal{O}_K$. If $\mathcal{Z}$ dominates $\Spec \mathcal{O}_K$, then $\mathcal{Z}_\eta$ is an irreducible component of $\mathcal{V}_\eta \cap \mathcal{H}_\eta$. It follows from Lemma \ref{lem:dimandcodim}~(1) and \cite[Proposition 5.40 and Corollary 5.42]{GoertzWedhorn} that $\dim \mathcal{Z} = \dim \mathcal{Z}_\eta + 1 = \dim \mathcal{V}_\eta = \dim \mathcal{V}-1$.

Suppose now that $\mathcal{Z} \subseteq \mathbb{P}^n_{\mathbb{F}_{\mathfrak{P}}}$ for some maximal ideal $\mathfrak{P}$ of $\mathcal{O}_K$. Choose a finite set $S$ of maximal ideals of $\mathcal{O}_K$ such that the ring of $S$-integers $\mathcal{O}_{K,S}$ is a unique factorization domain and $\mathfrak{P} \not\in S$.
Let $\mathbb{A}^n_{\mathcal{O}_{K,S}} \simeq U \subseteq \mathbb{P}^n_{\mathcal{O}_{K,S}}$ be a standard open affine subset that intersects $\mathcal{Z}$. By Lemma \ref{lem:degreereduction}~(3), we have $\dim \mathcal{H} \cap U = \dim \mathcal{H} = n$ and $\dim U = \dim \mathbb{P}^n_{\mathcal{O}_K} = n+1$. Therefore, we have $\codim_{U}{\mathcal{H} \cap U} = 1$ by Lemma \ref{lem:dimandcodim}~(3). By \cite[Proposition 5.31]{GoertzWedhorn}, $\mathcal{H} \cap U$ is the zero set of a single regular function on $U$. It then follows from \cite[Proposition 5.35]{GoertzWedhorn} that $\codim_{\mathcal{V} \cap U}{\mathcal{Z} \cap U} \leq 1$. Since $\mathcal{Z} \cap U \subseteq \mathcal{V} \cap U \cap \mathbb{P}^n_{\mathbb{F}_{\mathfrak{P}}} \subsetneq \mathcal{V} \cap U$, this implies that $\mathcal{Z}$ is an irreducible component of $\mathcal{V} \cap \mathbb{P}^n_{\mathbb{F}_{\mathfrak{P}}}$ and so $\dim \mathcal{Z} = \dim \mathcal{V}-1$ by Lemma \ref{lem:degreereduction}~(2).

Thus, each irreducible component of $\mathcal{V} \cap \mathcal{H}$ has dimension $\dim \mathcal{V}-1$ and is therefore proper as defined in \cite[Example 20.2.2]{Fulton}. Hence, the intersection cycle $\mathcal{V}.\mathcal{H}$ in \cite[Section 2.2]{BGS} (see also \cite[Chapters 8 and 20]{Fulton}) is a linear combination of the irreducible components of $\mathcal{V} \cap \mathcal{H}$, endowed with the reduced scheme structures, with coefficients in $\mathbb{N}$, see \cite[Example 20.2.2]{Fulton}. The theorem then follows from \cite[Theorem 4.2.3]{BGS} and Remark \ref{rmk:heightnonnegative}.
\end{proof}

In certain cases, the height is easy to estimate:

\begin{prop}\label{prop:heightestimates}
There exists a constant $c = c(K,n)$ such that the following holds: suppose that $\mathcal{Z}_i$ ($i = 1,\hdots,m$) are pairwise distinct irreducible components, endowed with the reduced scheme structure, of the closed subscheme of $\mathbb{P}^n_{\mathcal{O}_K}$ defined by a homogeneous polynomial $P \in \mathcal{O}_K[X_0,\hdots,X_n]\backslash\{0\}$. Suppose furthermore that each $\mathcal{Z}_i$ dominates $\Spec \mathcal{O}_K$ ($i = 1, \hdots, m$). Then \[ \sum_{i=1}^{m}{h(\mathcal{Z}_i)} \leq [K:\mathbb{Q}]h(P) + c\deg P,\]
where $h(P)$ denotes the usual absolute logarithmic Weil height of the vector of coefficients of $P$, seen as a point in projective space.
\end{prop}

\begin{proof}
Let $Z_i$ denote the generic fiber of $\mathcal{Z}_i$ ($i = 1,\hdots,m$). It follows from \cite[p.~346]{Philippon95} and \cite[Proposition 4.1.2]{BGS} that there exists a constant $c' = c'(K,n)$ such that
\begin{equation}\label{eq:heightcomparison}
    |[K:\mathbb{Q}]h(Z_i) - h(\mathcal{Z}_i)| \leq c'\deg \mathcal{Z}_i,
\end{equation}
where $i = 1, \hdots, m$ and $h(Z_i)$ is defined as in \cite{Philippon95}.

Let $P_i$ be a monic irreducible factor of $P$ in $K[X_0,\hdots,X_n]$ such that $Z_i$ is the hypersurface defined by $P_i$ ($i  = 1,\hdots,m$).
Since the $\mathcal{Z}_i$ are pairwise distinct, so are the $P_i$. It follows from the formula in \cite[top of p. 347]{Philippon95} that there exists a constant $c'' = c''(n)$ such that $h(Z_i) = h_{\mathrm{Phil}}(P_i) + c''\deg P_i$ ($i = 1,\hdots,m$), where $h_{\mathrm{Phil}}$ denotes the height of a homogeneous form used in \cite{Philippon95}. Note that $\sum_{i=1}^{m}{\deg P_i} \leq \deg P$ and $\sum_{i=1}^{m}{h_{\mathrm{Phil}}(P_i)} \leq  h_{\mathrm{Phil}}(P)$ thanks to the pairwise distinctness of the $P_i$ and the fact that $h_{\mathrm{Phil}}$ is non-negative and sends products of non-zero forms to sums of heights. Furthermore, we can bound $h_{\mathrm{Phil}}(P)$ from above by $h(P)+c''' \deg P$ for a constant $c''' = c'''(n)$. Hence, the proposition follows from \eqref{eq:heightcomparison} and the fact that $\sum_{i=1}^{m}{\deg \mathcal{Z}_i} = \sum_{i=1}^{m}{\deg P_i} \leq \deg P$.
\end{proof}

\subsection{Proof of Theorem \ref{thm:LIB}} \label{subsec:proof_LIB}

We return to the setting of Section \ref{subsec:prooftrivialboundnomult}. If $\mathcal{W} \subseteq \mathbb{G}^n_{m,\mathcal{O}_K}$ is an integral closed subscheme and $\overline{\mathcal{W}}$ denotes its schematic closure in $\mathbb{P}^{n}_{\mathcal{O}_K}$ (which is reduced), we set $h(\mathcal{W}) = h(\overline{\mathcal{W}})$ and $\deg \mathcal{W} = \deg \overline{\mathcal{W}}$ where $h$ and $\deg$ are defined as in Section \ref{subsec:arithmeticintersectiontheory}. The following proposition will be useful in the proof of Theorem \ref{thm:LIB}. In its proof, we will use the arithmetic intersection theory from Section \ref{subsec:arithmeticintersectiontheory}.

\begin{prop}\label{prop:inductivestep}
There exists a constant $C' = C'(K,n) \geq 1$ such that the following holds: let $\mathcal{W} \subseteq \mathbb{G}^n_{m,\mathcal{O}_K}$ be an integral closed subscheme and suppose that $\mathcal{H} \subseteq \mathbb{G}^n_{m,\mathcal{O}_K}$ is a flat subgroup scheme which does not contain $\mathcal{W}$. There is a finite set $S$ of integral closed subschemes $\mathcal{Z} \subsetneq \mathcal{W}$ such that each irreducible component of $\mathcal{W} \cap \mathcal{H}$, endowed with the reduced scheme structure, is contained in some $\mathcal{Z} \in S$ and $\sum_{\mathcal{Z} \in S}{(h(\mathcal{Z})+\deg \mathcal{Z})} \leq C'(h(\mathcal{W})+\deg \mathcal{W})\mathcal{C}(\mathcal{H})$.
\end{prop}

\begin{proof}
Let $x_1,\hdots,x_n$ denote the affine coordinates on $\mathbb{G}^n_{m,\mathcal{O}_K}$. It follows from $\mathcal{W} \not\subseteq \mathcal{H}$ and the definition of the complexity that there exist $a_i \in \mathbb{Z}$ such that $P(x_1,\hdots,x_n) = \prod_{i=1}^{n}{x_i^{a_i}}-1$ vanishes identically on $\mathcal{H}$, but not on $\mathcal{W}$, and $|a_i| \leq \mathcal{C}(\mathcal{H})$ ($i = 1,\hdots,n$).

Let now $X_0,X_1, \hdots,X_n$ denote the projective coordinates on $\mathbb{P}^n_{\mathcal{O}_K}$. Set 
\[a = \max\left\{\sum_{a_i > 0}{a_i},\sum_{a_i < 0}{-a_i}\right\}\]
and
\[ Q(X_0,\hdots,X_n) = X_0^{\max\{0,a-\sum_{a_i>0}{a_i}\}}\prod_{a_i > 0}{X_i^{a_i}} - X_0^{\max\{0,a+\sum_{a_i<0}{a_i}\}}\prod_{a_i < 0}{X_i^{-a_i}}.\]
The homogeneous polynomial $Q$ of degree $a \leq n\mathcal{C}(\mathcal{H})$ vanishes identically on the Zariski closure $\overline{\mathcal{H}}$ of $\mathcal{H}$ in $\mathbb{P}^{n}_{\mathcal{O}_K}$, but not on the Zariski closure $\overline{\mathcal{W}}$ of $\mathcal{W}$ in $\mathbb{P}^{n}_{\mathcal{O}_K}$. Let $\mathcal{Q}$ denote the closed subscheme of $\mathbb{P}^{n}_{\mathcal{O}_K}$ defined by $Q$.

Our set $S$ will be the set of irreducible components of the intersection of $\mathcal{W}$ with the closed subscheme of $\mathbb{G}^n_{m,\mathcal{O}_K}$ defined by $P$, where each irreducible component is endowed with the reduced scheme structure. It is then clear that each irreducible component of $\mathcal{W} \cap \mathcal{H}$ is contained in some $\mathcal{Z} \in S$. If $\mathcal{Z} \in S$, then its Zariski closure $\overline{\mathcal{Z}}$ in $\mathbb{P}^{n}_{\mathcal{O}_K}$ is an irreducible component of  $\overline{\mathcal{W}} \cap \mathcal{Q}$. 

To bound $\sum_{\mathcal{Z} \in S}{(h(\mathcal{Z})+\deg \mathcal{Z})}$, we distinguish two cases: first, we treat the case where $\mathcal{W}$ is contained in a closed fiber $\mathbb{G}^n_{m,\mathbb{F}_{\mathfrak{P}}}$ for some maximal ideal $\mathfrak{P}$ of $\mathcal{O}_K$. Then it follows from Remark \ref{rmk:heightinclosedfiber} together with Lemma \ref{lem:degreereduction}~(2) and the theorem of B\'ezout \cite[Example 8.4.6]{Fulton} that
\[\sum_{\mathcal{Z} \in S}{(h(\mathcal{Z})+\deg \mathcal{Z})} \leq (h(\mathcal{W})+\deg \mathcal{W})\deg Q \leq (h(\mathcal{W})+\deg \mathcal{W})n\mathcal{C}(\mathcal{H}).\]
Thus, we just have to take $C' \geq n$ in this case.

We now treat the case where $\mathcal{W}$ dominates $\Spec \mathcal{O}_K$: let $\mathcal{Q}'$ be an irreducible component of $\mathcal{Q}$, endowed with the reduced scheme structure. If an irreducible component $\mathcal{Z}$ of $\overline{\mathcal{W}} \cap \mathcal{Q}'$ is contained in a closed fiber $\mathbb{P}^{n}_{\mathbb{F}_{\mathfrak{P}}}$ for some maximal ideal $\mathfrak{P}$ of $\mathcal{O}_K$, then $\deg \mathcal{Z} = (\log N(\mathfrak{P}))^{-1} h(\mathcal{Z}) \leq 2h(\mathcal{Z})$. Together with B\'ezout's theorem \cite[Example 8.4.6]{Fulton}, applied on the generic fiber, and the fact that all the degrees and heights are non-negative (see Remark \ref{rmk:heightnonnegative}), this implies that
\begin{equation}\label{eq:bounddegreebyheight}
\sum_{\mathcal{Z}}{(h(\mathcal{Z})+\deg \mathcal{Z})} \leq \deg\mathcal{W}\deg\mathcal{Q}'+3\sum_{\mathcal{Z}}{h(\mathcal{Z})}
\end{equation}
where both sums run over the irreducible components $\mathcal{Z}$ of $\overline{\mathcal{W}} \cap \mathcal{Q}'$, endowed with the reduced scheme structure.

If $\mathcal{Q}'$ did not dominate $\Spec \mathcal{O}_K$, then $\mathcal{Q}' \subseteq \mathbb{P}^{n}_{\mathbb{F}_{\mathfrak{P}}}$ for some maximal ideal $\mathfrak{P}$ of $\mathcal{O}_K$ and in fact $\mathcal{Q}' \subsetneq \mathbb{P}^{n}_{\mathbb{F}_{\mathfrak{P}}}$ because the reduction of $Q$ modulo $\mathfrak{P}$ is non-zero. It follows that $\codim_{\mathbb{P}^n_{\mathcal{O}_K}}{\mathcal{Q}'} \geq 2$. But this contradicts \cite[Proposition 5.35]{GoertzWedhorn}.

Hence, $\mathcal{Q}'$ dominates $\Spec \mathcal{O}_K$. We therefore have $\dim \mathcal{Q}' = n$ by Lemma \ref{lem:degreereduction}~(2) and \cite[Proposition 5.40]{GoertzWedhorn}, applied on the generic fiber. Furthermore, $\overline{\mathcal{W}} \not\subseteq \mathcal{Q}'$. Thus, we can apply Theorem \ref{thm:arithmeticbezout} to deduce that
\begin{equation}\label{eq:arithmeticbezout}
\sum_{\mathcal{Z}}{h(\mathcal{Z})} \leq h(\mathcal{W})\deg \mathcal{Q}'+h(\mathcal{Q}')\deg \mathcal{W}+c\deg \mathcal{Q}'\deg \mathcal{W}
\end{equation}
where the sum runs over the irreducible components $\mathcal{Z}$ of $\overline{\mathcal{W}} \cap \mathcal{Q}'$, endowed with the reduced scheme structure, and $c = c([K:\mathbb{Q}],n)$ is a constant that depends only on $[K:\mathbb{Q}]$ and $n$. We can assume without loss of generality that $c \geq 1$.

We now combine \eqref{eq:bounddegreebyheight} with \eqref{eq:arithmeticbezout} and sum over all irreducible components $\mathcal{Q}'$ of $\mathcal{Q}$. We get
\begin{equation}\label{eq:boundheightsdegreesz}
    \sum_{\mathcal{Z}}{(h(\mathcal{Z}) + \deg \mathcal{Z})} \leq (3c+1)(h(\mathcal{W})+\deg \mathcal{W})\sum_{\mathcal{Q}'}{(h(\mathcal{Q}')+\deg \mathcal{Q}')}
\end{equation}
where now the sums run over all irreducible components $\mathcal{Z}$ of $\overline{\mathcal{W}} \cap \mathcal{Q}$, endowed with the reduced scheme structure, and all irreducible components $\mathcal{Q}'$ of $\mathcal{Q}$, endowed with the reduced scheme structure, respectively. Here, we have used again that the height is non-negative by Remark \ref{rmk:heightnonnegative}.

Since $\mathcal{Q}$ is defined by a homogeneous polynomial of degree at most $n\mathcal{C}(\mathcal{H})$, we have
\begin{equation}\label{eq:bounddegreesqprime}
    \sum_{\mathcal{Q}'}{\deg \mathcal{Q}'} \leq n\mathcal{C}(\mathcal{H}).
\end{equation}

It remains to bound $\sum_{\mathcal{Q}'}{h(\mathcal{Q}')}$ from above. Recall that every irreducible component $\mathcal{Q}'$ of $\mathcal{Q}$ dominates $\Spec \mathcal{O}_K$. It follows from Proposition  \ref{prop:heightestimates} that
\begin{equation}\label{eq:boundheightsqprime}
    \sum_{\mathcal{Q}'}{h(\mathcal{Q}')} \leq [K:\mathbb{Q}]h(Q) + c'\deg Q \leq [K:\mathbb{Q}]h(Q) + c'n \mathcal{C}(\mathcal{H})
\end{equation}
for a constant $c' = c'(K,n)$ where $h(Q)$ denotes the height of the vector of coefficients of $Q$, seen as a point in projective space. Since all coefficients of $Q$ belong to $\{0,\pm1\}$, we have $h(Q) = 0$. We now deduce the lemma by putting this together with \eqref{eq:boundheightsdegreesz}, \eqref{eq:bounddegreesqprime}, and \eqref{eq:boundheightsqprime}.
\end{proof}

We now prove Theorem \ref{thm:LIB}:

\begin{proof}[Proof of Theorem \ref{thm:LIB}]
Since $\dim \mathcal{V} > 0$ and therefore $\mathcal{H} \neq \mathbb{G}^n_{m,\mathcal{O}_K}$, we have $\mathcal{C}(\mathcal{H}) \geq 1$ by definition of the complexity. Let $C' = C'(K,n)$ denote the constant from Proposition \ref{prop:inductivestep}.

We will prove the following claim by induction on $k \in \{0,\hdots,\dim \mathcal{V}\}$: there exists a finite set $S_k$ of integral closed subschemes $\mathcal{Z} \subseteq \mathcal{V}$ such that:
\begin{enumerate}
    \item each irreducible component of $\mathcal{V} \cap \mathcal{H}$, endowed with the reduced scheme structure, is contained in some $\mathcal{Z} \in S_k$, 
    \item $\sum_{\mathcal{Z} \in S_k}{(h(\mathcal{Z})+\deg \mathcal{Z})} \leq C'^k(h(\mathcal{V})+\deg \mathcal{V})\mathcal{C}(\mathcal{H})^k$, and 
    \item either $\mathcal{Z}$ is an irreducible component of $\mathcal{V} \cap \mathcal{H}$, endowed with the reduced scheme structure, or $\dim \mathcal{Z} \leq \dim \mathcal{V}-k$.
\end{enumerate}

The base case $k = 0$ of the induction is trivial with $S_0 = \{\mathcal{V}\}$. Suppose that the claim holds for some $k \in \{0,\hdots,\dim \mathcal{V}-1\}$.

Let $\mathcal{Z} \in S_k$. We now define a set $S_{k+1,\mathcal{Z}}$ as follows:
if $\dim \mathcal{Z} \leq \dim \mathcal{V} - (k+1)$, we set $S_{k+1,\mathcal{Z}} = \{\mathcal{Z}\}$. If $\dim \mathcal{Z} = \dim \mathcal{V}-k$ and $\mathcal{Z} \subseteq \mathcal{H}$, then $\mathcal{Z} \subseteq \mathcal{V} \cap \mathcal{H}$. If $\mathcal{Z}$ does not contain any irreducible component of $\mathcal{V} \cap \mathcal{H}$, endowed with the reduced scheme structure, we set $S_{k+1,\mathcal{Z}} = \emptyset$. Otherwise, it follows that $\mathcal{Z}$ is equal to some irreducible component of $\mathcal{V} \cap \mathcal{H}$, endowed with the reduced scheme structure, and we set $S_{k+1,\mathcal{Z}} = \{\mathcal{Z}\}$. 

It remains to define $S_{k+1,\mathcal{Z}}$ in the case where $\dim \mathcal{Z} = \dim \mathcal{V}-k$ and $\mathcal{Z} \not \subseteq \mathcal{H}$.
We apply Proposition \ref{prop:inductivestep} with $\mathcal{W} = \mathcal{Z}$ and set $S_{k+1,\mathcal{Z}}$ to be the set $S$ of integral closed subschemes that Proposition \ref{prop:inductivestep} yields. 
Each element of $S$ has dimension at most $\dim \mathcal{Z}-1 = \dim \mathcal{V} - (k+1)$. 
Furthermore, each irreducible component of $\mathcal{V} \cap \mathcal{H}$ that is contained in $\mathcal{Z}$, endowed with the reduced scheme structure, is contained in some element of $S_{k+1,\mathcal{Z}}$ and
\[ \sum_{\mathcal{Y} \in S_{k+1,\mathcal{Z}}}{(h(\mathcal{Y}) + \deg \mathcal{Y})} \leq C'(h(\mathcal{Z})+\deg \mathcal{Z})\mathcal{C}(\mathcal{H}).\]

The induction step is now complete with
\[S_{k+1} = \bigcup_{\mathcal{Z} \in S_k}{S_{k+1,\mathcal{Z}}}\]
since $C' \geq 1$, $\mathcal{C}(\mathcal{H}) \geq 1$, and all heights and degrees are non-negative thanks to Remark \ref{rmk:heightnonnegative}. 

Finally, the claim for $k = \dim \mathcal{V}$ together with Remarks \ref{rmk:heightinclosedfiber} and \ref{rmk:heightnonnegative} implies Theorem \ref{thm:LIB}.
\end{proof}

\begin{rmk}
The proof of Theorem \ref{thm:LIB} directly yields the stronger conclusion that
\[\sum_{\mathcal{Z}}{h(\mathcal{Z}_{\red})} \leq C\mathcal{C}(\mathcal{H})^{\dim \mathcal{V}},\]
where the sum on the left-hand side now runs over all irreducible components of $\mathcal{V} \cap \mathcal{H}$ (including those that dominate the base). Note also that the proof does not really use the fact that $\dim \mathcal{V} +\dim \mathcal{H} \leq n+1$, but only that $\mathcal{H} \neq \mathbb{G}_{m,\mathcal{O}_K}^n$. If $\dim \mathcal{V} +\dim \mathcal{H} > n+1$, then the exponent of $\mathcal{C}(\mathcal{H})$ can even be improved to $\codim \mathcal{H}$.
\end{rmk}

\subsection{Application to multiplicative dependence with two independent relations modulo primes} \label{subsec:BCMOS}

We continue working in the setting of Section \ref{subsec:prooftrivialboundnomult}. In \cite[Theorem 2.1]{BCMOS_Preprint}, Barroero, Capuano, M\'erai, Ostafe, and Sha consider points on the reduction modulo a prime of a curve in $\mathbb{G}^n_{m,\mathbb{Q}}$ parametrized by non-zero multiplicatively independent rational functions. They show that, if the prime is large enough, such a point does not satisfy two independent multiplicative relations with small exponents unless it is the reduction of a point on the curve that satisfies two independent multiplicative relations in characteristic $0$. We now show how Proposition \ref{prop:inductivestep} from Section \ref{subsec:proof_LIB} can be used to prove a generalization of this result:

\begin{thm}\label{cor:curveingm}
Let $\mathcal{V} \subseteq \mathbb{G}^{n}_{m,\mathcal{O}_K}$ be a $2$-dimensional integral closed subscheme that dominates $\Spec \mathcal{O}_K$. Suppose that $\mathcal{V}_\eta$ is not contained in any proper algebraic subgroup of $\mathbb{G}^{n}_{m,K}$.

There exist a finite set $S = S(\mathcal{V})$ of $1$-dimensional integral closed subschemes $\mathcal{X} \subseteq \mathcal{V}$, flat over $\Spec \mathcal{O}_K$, and a constant $C = C(K,n,\mathcal{V})$ such that the following holds: let $(k_1,\hdots,k_n)$ and $(l_1,\hdots,l_n)$ be two linearly independent integer vectors, set $M = (\max_{i}{|k_i|})(\max_{i}{|l_i|})$, and let $\mathcal{H} \subseteq \mathbb{G}^n_{m,\mathcal{O}_K}$ be the subgroup scheme defined by the equations $\prod_{i=1}^{n}{X_i^{k_i}} = \prod_{i=1}^{n}{X_i^{l_i}} = 1$ in the affine coordinates $X_1,\hdots,X_n$ on $\mathbb{G}^n_{m,\mathcal{O}_K}$.

Then every irreducible component $\mathcal{Z}$ of $\mathcal{V} \cap \mathcal{H}$ that dominates $\Spec \mathcal{O}_K$ belongs to $S$ and furthermore
\[ \sum_{\mathcal{Z}}{\deg(\mathcal{Z}_{\red})\log N(\mathcal{Z})} \leq CM,\]
where the sum runs over all irreducible components $\mathcal{Z}$ of $\mathcal{V} \cap \mathcal{H}$ that are not contained in an element of $S$ and $\deg(\mathcal{Z}_{\red})$ and $N(\mathcal{Z})$ are as defined in Section \ref{subsec:prooftrivialboundnomult}.
\end{thm}

Compared with \cite[Theorem 2.1]{BCMOS_Preprint}, Theorem \ref{cor:curveingm} removes the restriction that the curve is parametrized by a vector $\boldsymbol{\varphi}$ of rational functions with rational coefficients and gives a bound for $\sum_{\mathcal{Z}}{\deg(\mathcal{Z}_{\red})\log N(\mathcal{Z})}$ instead of $\max_{\mathcal{Z}} \log N(\mathcal{Z})$. Note that our set $S$ essentially corresponds to the set $\mathcal{S}_1$ in \cite[Theorem 2.1]{BCMOS_Preprint} in the sense that $\mathcal{X} \in S$ if and only if the geometric generic fiber of $\mathcal{X}$ consists of images of elements of $\mathcal{S}_1$ under $\boldsymbol{\varphi}$.

The linear dependence on $M$ of the upper bound in Theorem \ref{cor:curveingm} is best possible as the example $n = 2$, $\mathcal{V}$ equal to the Zariski closure of $\{2\} \times_{K} \mathbb{G}_{m,K}$ in $\mathbb{G}^2_{m,\mathcal{O}_K}$, $k_2 = l_1 = 0$, $k_1 = 2$, and $l_2 \in \mathbb{N}$ not a multiple by $3$ shows, where the left-hand side evaluates to $l_2\log 3 = M(\log 3)/2$.

\begin{proof}[Proof of Theorem \ref{cor:curveingm}]
By \cite[Th\'eor\`eme 1.2]{Maurin08}, the intersection of $\mathcal{V}_\eta$ with the union of all algebraic subgroups of $\mathbb{G}^n_{m,K}$ of codimension $\geq 2$ is a finite set $X$. We set $S$ equal to the set of irreducible components of the Zariski closure of $X$ in $\mathbb{G}^n_{m,\mathcal{O}_K}$, each component endowed with the reduced scheme structure. By \cite[Proposition 14.14]{GoertzWedhorn}, all these components are flat over $\Spec \mathcal{O}_K$. Furthermore, they are all $1$-dimensional by Lemma \ref{lem:dimandcodim}~(1).

Let $\mathcal{H}$ be as in the statement of Theorem \ref{cor:curveingm}. We have that $\mathcal{V} \cap \mathcal{H} \subsetneq \mathcal{V}$ and that every irreducible component of $\mathcal{V} \cap \mathcal{H}$ that dominates $\Spec \mathcal{O}_K$, endowed with the reduced scheme structure, belongs to $S$. We apply Proposition \ref{prop:inductivestep} to $\mathcal{W} = \mathcal{V}$ and the flat subgroup scheme $\mathcal{H}_1$ of $\mathbb{G}^n_{m,\mathcal{O}_K}$ defined by the equation $\prod_{i=1}^{n}{X_i^{k_i}} = 1$. We have $\mathcal{V} \not\subseteq \mathcal{H}_1$ since $\mathcal{V}_\eta$ is not contained in any proper algebraic subgroup of $\mathbb{G}^{n}_{m,K}$. We get a finite set $S_0$ of integral closed subschemes $\mathcal{Z}_0 \subsetneq \mathcal{V}$ such that
\begin{equation}\label{eq:aramis}
\sum_{\mathcal{Z}_0 \in S_0}{(h(\mathcal{Z}_0)+\deg \mathcal{Z}_0)} \leq C'(h(\mathcal{V})+\deg \mathcal{V})\mathcal{C}(\mathcal{H}_1)
\end{equation}
for the constant $C' = C'(K,n) $ from Proposition \ref{prop:inductivestep}. Furthermore,
\[ \mathcal{Z}_{\red} \subseteq \bigcup_{\mathcal{Z}_0 \in S_0}{\mathcal{Z}_0}\]
for each irreducible component $\mathcal{Z}$ of $\mathcal{V} \cap \mathcal{H}_1$ and we may and will assume that each $\mathcal{Z}_0 \in S_0$ contains at least one such irreducible component $\mathcal{Z}_{\red}$. But, by \cite[Proposition 5.35]{GoertzWedhorn} and Lemma \ref{lem:dimandcodim}~(3), every such irreducible component is $1$-dimensional and so, again by Lemma \ref{lem:dimandcodim}~(3), every element of $S_0$ is a $1$-dimensional irreducible component of $\mathcal{V} \cap \mathcal{H}_1$, endowed with the reduced scheme structure.

We now partition $S_0$ into four sets $S_1, S_2, S_3, S_4$: the set $S_1$ contains all elements of $S_0$ that do not dominate $\Spec \mathcal{O}_K$ and are irreducible components of $\mathcal{V} \cap \mathcal{H}$, endowed with the reduced scheme structure; the set $S_2$ contains all elements of $S_0$ that do not dominate $\Spec \mathcal{O}_K$ and are not irreducible components of $\mathcal{V} \cap \mathcal{H}$, endowed with the reduced scheme structure; finally, we set $S_3 = S \cap S_0$ and $S_4 = S_0\backslash(S_1 \cup S_2 \cup S_3)$. We have
\[ \mathcal{Z}_{\red} \subseteq \bigcup_{\mathcal{Z}_0 \in S_1 \cup S_2 \cup S_4}{\mathcal{Z}_0}\]
for each irreducible component $\mathcal{Z}$ of $\mathcal{V} \cap \mathcal{H}$ that is not contained in an element of $S$.

Let now $\mathcal{H}_2$ denote the flat subgroup scheme of $\mathbb{G}^n_{m,\mathcal{O}_K}$ defined by the equation $\prod_{i=1}^{n}{X_i^{l_i}} = 1$ and let $\mathcal{Z} \in S_4$ be arbitrary. If $\mathcal{Z} \subseteq \mathcal{H}_2$, then $\mathcal{Z}$ must be an irreducible component of $\mathcal{V} \cap \mathcal{H}$, endowed with the reduced scheme structure, since $\mathcal{Z}$ is an irreducible component of $\mathcal{V} \cap \mathcal{H}_1$ and $\mathcal{H} = \mathcal{H}_1 \cap \mathcal{H}_2$. Moreover, the scheme $\mathcal{Z}$ dominates $\Spec \mathcal{O}_K$. Hence $\mathcal{Z} \in S_3$ in contradiction with $\mathcal{Z} \in S_4$. 

Thus, we must have $\mathcal{Z} \not\subseteq \mathcal{H}_2$. We now apply Proposition \ref{prop:inductivestep} to $\mathcal{W} = \mathcal{Z}$ and $\mathcal{H}_2$. We get a finite set $S_{4,\mathcal{Z}}$ of integral closed subschemes $\mathcal{Y} \subsetneq \mathcal{Z}$ such that each irreducible component of $\mathcal{Z} \cap \mathcal{H}_2$, endowed with the reduced scheme structure, is contained in some $\mathcal{Y} \in S_{4,\mathcal{Z}}$. Since $\mathcal{Y} \subsetneq \mathcal{Z} \subsetneq \mathcal{V}$, we must have $\dim \mathcal{Y} = 0$ for all $\mathcal{Y} \in S_{4,\mathcal{Z}}$ by Lemma \ref{lem:dimandcodim}~(3). Furthermore,
\begin{equation}\label{eq:athos}
\sum_{\mathcal{Y} \in S_{4,\mathcal{Z}}}{(h(\mathcal{Y})+\deg \mathcal{Y})} \leq C'(h(\mathcal{Z})+\deg \mathcal{Z})\mathcal{C}(\mathcal{H}_2).
\end{equation}

We can apply the same procedure to any $\mathcal{Z} \in S_2$: recall that $\mathcal{Z}$ is a $1$-dimensional irreducible component of $\mathcal{V} \cap \mathcal{H}_1$, endowed with the reduced scheme structure, but $\mathcal{Z}$ is not an irreducible component of $\mathcal{V} \cap \mathcal{H}$, endowed with the reduced scheme structure. Therefore $\mathcal{Z} \not\subseteq \mathcal{H}_2$ and we can apply Proposition \ref{prop:inductivestep} as we did above. We get a finite set $S_{2,\mathcal{Z}}$ of integral closed subschemes $\mathcal{Y} \subsetneq \mathcal{Z}$ such that each irreducible component of $\mathcal{Z} \cap \mathcal{H}_2$, endowed with the reduced scheme structure, is contained in some $\mathcal{Y} \in S_{4,\mathcal{Z}}$ and such that
\begin{equation}\label{eq:porthos}
\sum_{\mathcal{Y} \in S_{2,\mathcal{Z}}}{(h(\mathcal{Y})+\deg \mathcal{Y})} \leq C'(h(\mathcal{Z})+\deg \mathcal{Z})\mathcal{C}(\mathcal{H}_2).
\end{equation}
Again, it follows from $\mathcal{Y} \subsetneq \mathcal{Z} \subsetneq \mathcal{V}$ and Lemma \ref{lem:dimandcodim}~(3) that $\dim  \mathcal{Y} = 0$.

We set
\[ S_5 = S_1 \cup \bigcup_{\mathcal{Z} \in S_2}{S_{2,\mathcal{Z}}} \cup  \bigcup_{\mathcal{Z} \in S_4}{S_{4,\mathcal{Z}}}.\]
By construction, we have
\[ \mathcal{Z}_{\red} \subseteq \bigcup_{\mathcal{Z}_1 \in S_5}{\mathcal{Z}_1}\]
for each irreducible component $\mathcal{Z}$ of $\mathcal{V} \cap \mathcal{H}$ that is not contained in an element of $S$. But since all elements of $S_5$ are either irreducible components of $\mathcal{V} \cap \mathcal{H}$, endowed with the reduced scheme structure, or $0$-dimensional, it follows that for each irreducible component $\mathcal{Z}$ of $\mathcal{V} \cap \mathcal{H}$ that is not contained in an element of $S$, $\mathcal{Z}_{\red}$ is an element of $S_5$. Together with \eqref{eq:aramis}, \eqref{eq:athos}, \eqref{eq:porthos}, and the fact that $1 \leq \mathcal{C}(\mathcal{H}_1)\mathcal{C}(\mathcal{H}_2) \leq M$ (by definition of the complexity), this implies that
\[ \sum_{\mathcal{Z}}{h(\mathcal{Z}_{\red})} \leq C'^2(h(\mathcal{V})+\deg \mathcal{V})M,\]
where the sum runs over all irreducible components $\mathcal{Z}$ of $\mathcal{V} \cap \mathcal{H}$ that are not contained in an element of $S$. The corollary now follows from Remark \ref{rmk:heightinclosedfiber}.
\end{proof}

\section{Bounding the size of arithmetic unlikely intersections}\label{sec:mainconj}

Motivated by Theorem \ref{thm:LIB}, we formulate in this section a general conjecture about arithmetic unlikely intersections in powers of the arithmetic multiplicative group, which encompasses results of Bugeaud-Corvaja-Zannier \cite{Bugeaud_Corvaja_Zannier_2003} and Corvaja-Zannier \cite{Corvaja_Zannier_2005}. The following notation has been introduced in Sections \ref{sec:complexity} and \ref{sec:trivialbound} and will be in force for the rest of this section: let $K$ be a number field, let $\eta$ denote the generic point of $\Spec \mathcal{O}_K$, and let $n \in \mathbb{N}$. For $\mathcal{Z} \subseteq \mathbb{G}^n_{m,\mathcal{O}_K}$ an irreducible closed subscheme that does not dominate $\Spec \mathcal{O}_K$, we refer to Section \ref{subsec:prooftrivialboundnomult} for the definition of $N(\mathcal{Z})$ and $\deg \mathcal{Z}_{\red}$.

\begin{conj}[Unlikely Intersection Bound] \label{conj:mainconj}
Let $\mathcal{V} \subseteq \mathbb{G}_{m,\mathcal{O}_K}^n$ be an integral closed subscheme that dominates $\Spec \mathcal{O}_K$. Suppose that $V := \mathcal{V}_{\eta}$ is not contained in any proper algebraic subgroup of $\mathbb{G}^n_{m,K}$. For every $\varepsilon > 0$, there exists a constant $c = c(n,\mathcal{V},\varepsilon, K) \in \mathbb{R}$ such that for all flat subgroup schemes $\mathcal{H} \subseteq \mathbb{G}_{m,\mathcal{O}_K}^n$ which satisfy $\dim \mathcal{H} + \dim \mathcal{V} < n+1$ we have

\begin{equation}\label{eq:inequalityconjecture}
    \sum_{\mathcal{Z}}{\mult(\mathcal{Z})\deg(\mathcal{Z}_{\mathrm{\red}})\log N(\mathcal{Z})} \leq \varepsilon\mathcal{C}(\mathcal{H})^{\dim \mathcal{V}} + c,
\end{equation}
where the sum runs over all irreducible components $\mathcal{Z}$ of $\mathcal{V} \cap \mathcal{H}$ which do not dominate $\Spec \mathcal{O}_K$ and $\mult(\mathcal{Z})$ is as defined in Definition \ref{defn:multiplicity}.
\end{conj}

We will refer to the upper bound in this conjecture as the Unlikely Intersection Bound, abbreviated ULIB.

In Conjectures \ref{conj:LIB} and \ref{conj:mainconj}, we consider only those irreducible components of $\mathcal{V} \cap \mathcal{H}$ that do not dominate the base since we want to focus on intersections that are ``properly" arithmetic and this is the simplest way of doing so. According to the Zilber-Pink conjecture, \emph{i.e.} \cite[Conjecture 2]{Zilber}, \cite[Conjecture 1.1]{PinkUnpubl}, and \cite[Torsion Openness Conjecture and Torsion Finiteness Conjecture]{BMZ07}, we expect an open dense subscheme $\mathcal{U}$ of $\mathcal{V}$ to exist such that, for each flat subgroup scheme $\mathcal{H} \subseteq \mathbb{G}^n_{m,\mathcal{O}_K}$ which satisfies $\dim \mathcal{H} + \dim \mathcal{V} < n+1$, no irreducible component of $\mathcal{U} \cap \mathcal{H}$ dominates $\Spec \mathcal{O}_K$. We have not formulated Conjecture \ref{conj:mainconj} in this weaker form since it would make explicit computations much more complicated and the numerical evidence we have also supports the stronger Conjecture \ref{conj:mainconj}.

If $V$ is contained in a proper algebraic subgroup of $\mathbb{G}^n_{m,K}$, then \eqref{eq:inequalityconjecture} becomes false in general as the following example shows.

\begin{ex}
Take $n=2$, $\mathcal{H} = \ker [N]$ ($N \in \mathbb{N}$), and $\mathcal{V}$ equal to the Zariski closure of $\{(2,4)\} \subseteq \mathbb{G}^2_{m,\mathbb{Q}}(\mathbb{Q})$ inside $\mathbb{G}^2_{m,\mathbb{Z}}$. Then, by Remark \ref{rmk:lhszerodim}, Conjecture \ref{conj:mainconj} for $\mathcal{V}$ and $\mathcal{H}$ would imply the patently false inequality $\log(\gcd(2^N-1,4^N-1)) \leq N/2+c$ for some constant $c \in \mathbb{R}$ and all $N \in \mathbb{N}$.
\end{ex}

In contrast to the Zilber-Pink conjecture, we do not conjecture a uniform upper bound for the left-hand side of \eqref{eq:inequalityconjecture} that is independent of $\mathcal{H}$ since this would be false already because of the simple fact that every prime $p > 3$ divides $\gcd(2^{p-1}-1,3^{p-1}-1)$.

\subsection{The case $\dim \mathcal{V}=1$}
In this section, we show how some results from the literature, due to Corvaja and Zannier, imply Conjecture \ref{conj:mainconj} when $\dim \mathcal{V} = 1$. The following theorem is a direct consequence of their work. For the notation used in this theorem and its proof, we refer to Section \ref{subsec:generalities}.

\begin{thm}\label{thm:corvajazannier}
Let $\alpha_1, \hdots,\alpha_n \in K^{\ast}$ be multiplicatively independent and let $\varepsilon > 0$. Let $S$ denote the (finite) set of maximal ideals $\mathfrak{P}$ of $\mathcal{O}_K$ such that $\ord_{\mathfrak{P}}(\alpha_i) \neq 0$ for some $i \in \{1,\hdots,n\}$. Then there exists $c = c(K,\alpha_1,\hdots,\alpha_n,\varepsilon) > 0$ such that
\[    \log N(\gcd_{S}(\alpha_1^{k_1}\cdots \alpha_n^{k_n}-1,\alpha_1^{l_1}\cdots \alpha_n^{l_n}-1)) \leq \varepsilon \max_{i=1,\hdots,n}\{|k_i|,|l_i|\}+c\]
for all pairs of linearly independent vectors $(k_1, \hdots, k_n), (l_1, \hdots, l_n) \in \mathbb{Z}^n$.
\end{thm}

\begin{proof}
Set
\[ \varepsilon' = [K:\mathbb{Q}]^{-1}\left(\sum_{i=1}^{n}{h(\alpha_i)}\right)^{-1}\varepsilon.\]
Version 3 of \cite[Theorem 2.1]{ZannierBook} (which follows from the results in \cite{Corvaja_Zannier_2005}) says that there is a constant $c' = c'(K,S,\varepsilon')$ such that for multiplicatively independent $x,y \in \mathcal{O}_{K,S}^{\ast}$ we have:
\begin{equation}\label{eq:asymptotic}
    h_{\mathrm{gcd}}(x-1,y-1) \leq \varepsilon'\max\{h(x),h(y)\}+c'.
\end{equation}
Here $h_{\mathrm{gcd}}(x-1,y-1)$ is defined to be
\[\sum_{v \in M_K}{\min\{\log^+|(x-1)^{-1}|_v,\log^+|(y-1)^{-1}|_v\}}.\]

Since $\alpha_1,\hdots,\alpha_n$ are multiplicatively independent $S$-units and $(k_1, \hdots, k_n), (l_1, \hdots, l_n)$ are linearly independent, we may choose $x = \alpha_1^{k_1}\cdots \alpha_n^{k_n}$ and $y = \alpha_1^{l_1}\cdots \alpha_n^{l_n}$ as multiplicatively independent elements of $\mathcal{O}_{K,S}^{\ast}$.

If $\mathfrak{P}$ is a maximal ideal of $\mathcal{O}_K$ that does not lie in $S$, then
\[\min\{\log^+|(x-1)^{-1}|_{\mathfrak{P}},\log^+|(y-1)^{-1}|_{\mathfrak{P}}\} = [K:\mathbb{Q}]^{-1}\ord_{\mathfrak{P}}(\gcd_{S}(x-1,y-1))\log N(\mathfrak{P}).\]
Therefore, it follows from \eqref{eq:asymptotic} that
\[\log N(\gcd_{S}(x-1,y-1)) \leq [K:\mathbb{Q}](\varepsilon'\max\{h(x),h(y)\}+c').\]

The theorem now follows from
\[ \max\{h(x),h(y)\} \leq \max\left\{\sum_{i=1}^{n}{|k_i|h(\alpha_i)},\sum_{i=1}^{n}{|l_i|h(\alpha_i)}\right\} \leq \left(\sum_{i=1}^{n}{h(\alpha_i)}\right)\max_{i=1,\hdots,n}\{|k_i|,|l_i|\}.\]

\end{proof}

We now use Theorem \ref{thm:corvajazannier} to prove our Conjecture \ref{conj:mainconj} when $\dim \mathcal{V} = 1$.

\begin{thm} \label{thm:ULIB_dim1}
Conjecture \ref{conj:mainconj} holds if $\dim \mathcal{V} = 1$.
\end{thm}

\begin{proof}

Recall that $V = \mathcal{V}_\eta$. Both $\mathcal{V}$ and $V$ are affine. Since $\mathcal{V}$ is integral, the ring $\mathcal{O}_{\mathcal{V}}(\mathcal{V})$ is an integral domain. Since both $\mathcal{V}$ and $\Spec K$ are flat over $\Spec \mathcal{O}_K$ by \cite[Proposition 14.14]{GoertzWedhorn}, we have a natural inclusion $\mathcal{O}_{\mathcal{V}}(\mathcal{V}) \subseteq L:= \mathcal{O}_V(V) = \mathcal{O}_{\mathcal{V}}(\mathcal{V}) \otimes_{\mathcal{O}_K} K$ and we deduce that $L$ is an integral domain. Hence, by \cite[Corollary 5.21]{GoertzWedhorn}, the ring $L$ is a number field containing $K$.

By the integrality of $\mathcal{V}$, the ring $\mathcal{O}_{\mathcal{V}}(\mathcal{V})$ is a quotient of $\mathcal{O}_{K}[X_1^{\pm 1},\hdots,X_n^{\pm 1}]$ by a prime ideal. Let $\alpha_i$ denote the image of $X_i$ in $L^{\ast}$ for $i = 1,\hdots,n$, then we obtain an isomorphism
\[ \mathcal{V} \simeq \Spec \mathcal{O}_{K}[\alpha_1^{\pm 1},\hdots,\alpha_n^{\pm 1}],\]
where the $\alpha_i$ are multiplicatively independent since $V$ is not contained in any proper algebraic subgroup of $\mathbb{G}_{m,K}^n$.

Let $\mathcal{H}$ be any flat subgroup scheme such that $\dim \mathcal{H} + \dim \mathcal{V} < n+1$, \emph{i.e.} such that $\dim \mathcal{H} \leq n-1$. There exist linearly independent vectors $(k_1,\hdots,k_n)$, $(l_1,\hdots,l_n) \in \mathbb{Z}^n$ such that
\[ \mathcal{H} \subseteq \ker(X_1^{k_1}\cdots X_n^{k_n}, X_1^{l_1} \cdots X_n^{l_n})\]
and
\begin{equation}\label{eq:complexitykili}
\max_{i=1,\hdots,n}\{|k_i|,|l_i|\} \leq \mathcal{C}(\mathcal{H}).
\end{equation}

The scheme $\mathcal{V} \cap \mathcal{H}$ is isomorphic to $\Spec R_{\mathcal{H}}$, where $R_{\mathcal{H}}$ is a quotient of the ring
\[ \mathcal{O}_{K}[\alpha_1^{\pm 1},\hdots,\alpha_n^{\pm 1}]/J_\mathcal{H}\]
and $J_\mathcal{H}:= (\alpha_1^{k_1}\cdots \alpha_n^{k_n}-1,\alpha_1^{l_1}\cdots \alpha_n^{l_n}-1) $. Let $S$ denote the set of all maximal ideals $\mathfrak{p}$ of $\mathcal{O}_L$ such that $\ord_{\mathfrak{p}}(\alpha_i) \neq 0$ for some $i \in \{1,\hdots,n\}$. We find that $\mathcal{O}_{K}[\alpha_1^{\pm 1},\hdots,\alpha_n^{\pm 1}] \subseteq \mathcal{O}_{L,S}$. Since the ring $\mathcal{O}_{L,S'}$ is integrally closed for every set $S'$ of maximal ideals of $\mathcal{O}_L$ and $S$ is the minimal such set $S'$ with respect to inclusion such that $\mathcal{O}_{K}[\alpha_1^{\pm 1},\hdots,\alpha_n^{\pm 1}] \subseteq \mathcal{O}_{L,S'}$, it follows from \cite[Lemma 2.1 and the following paragraph and Theorem 2.7]{Beaumont_Pierce_1961} that
\[I := [\mathcal{O}_{L,S}:\mathcal{O}_{K}[\alpha_1^{\pm 1},\hdots,\alpha_n^{\pm 1}]] < \infty.\]

Since $J_\mathcal{H}$ is generated by two elements, we have 
\[
\left |\mathcal{O}_{K}[\alpha_1^{\pm 1},\hdots,\alpha_n^{\pm 1}]/J_\mathcal{H} \right| \leq  \left | \mathcal{O}_{L,S} /  J_\mathcal{H} \mathcal{O}_{L,S}  \right | \cdot \left | J_\mathcal{H} \mathcal{O}_{L,S} / J_\mathcal{H} \right | \leq  \left | \mathcal{O}_{L,S} /  J_\mathcal{H} \mathcal{O}_{L,S}  \right | \cdot I^2 < \infty.
\]
Thanks to Remark \ref{rmk:lhszerodim}, the left-hand side in Conjecture \ref{conj:mainconj} is therefore bounded from above by
\[ \log \left|\mathcal{O}_{L,S}/ J_\mathcal{H} \mathcal{O}_{L,S} \right|+2\log I,\]
where the second summand is now independent of $\mathcal{H}$ and the first summand equals
\[ \log N(\gcd_{S}(\alpha_1^{k_1}\cdots \alpha_n^{k_n}-1,\alpha_1^{l_1}\cdots \alpha_n^{l_n}-1)).\]
Conjecture \ref{conj:mainconj} now follows from Theorem \ref{thm:corvajazannier} and \eqref{eq:complexitykili}.
\end{proof}

\begin{rmk}\label{rmk:LIB_dim1}
Arguing along the same lines as in the proof of Theorem \ref{thm:ULIB_dim1}, one can reduce the $1$-dimensional case of Conjecture \ref{conj:LIB} to the bound
\[ \log N((\alpha_1^{k_1}\cdots \alpha_n^{k_n}-1)\mathcal{O}_{L,S}) \leq C\max_{i=1, \hdots, n} |k_i|,\]
which follows from 
\[ \log N((\alpha_1^{k_1}\cdots \alpha_n^{k_n}-1)\mathcal{O}_{L,S}) \leq [L:\mathbb{Q}]h(\alpha_1^{k_1}\cdots \alpha_n^{k_n}-1) \leq [L:\mathbb{Q}]\left(\log 2 + \sum_{i=1}^{n}{|k_i|h(\alpha_i)}\right).\]
\end{rmk}

\subsection{The case $\dim \mathcal{V}=2$}

Going beyond the case where $\dim \mathcal{V} = 1$, we will show in this section that when $\dim \mathcal{V} = 2$, $n \geq 3$, and $\mathcal{H} = \ker [N]$ for some $N \in \mathbb{N}$, then the terms $\log N(\mathcal{Z})$ appearing in the sum on the left-hand side of \eqref{eq:inequalityconjecture} cannot grow faster than linearly in $N = \mathcal{C}(\mathcal{H})$ (see Lemma \ref{lem:complexitymultbyN}). If $\mathcal{Z}$ does not arise as the reduction of an unlikely intersection in characteristic $0$, we even get a power saving. If furthermore $n = 2$, then the intersection $\mathcal{V} \cap \mathcal{H}$ is likely and Theorem \ref{cor:curveingm} gives an upper bound for $\log N(\mathcal{Z})$, but we also manage to improve on this upper bound for $\log N(\mathcal{Z})$, replacing $\mathcal{O}(N^2)$ with $\mathcal{O}(N^{3/2})$.

\begin{thm}\label{thm:singleprimecurve}
Suppose that $n \geq 2$. Let $\mathcal{V} \subseteq \mathbb{G}^n_{m,\mathcal{O}_K}$ be an integral closed subscheme of dimension $2$ that dominates $\Spec \mathcal{O}_K$. Suppose that $\mathcal{V}_{\eta}$ is not contained in any proper algebraic subgroup of $\mathbb{G}^n_{m,K}$. There exists $C = C(K,n,\mathcal{V})$ such that the following holds:

Suppose that $N \in \mathbb{N}$, $\mathfrak{P}$ is a maximal ideal of $\mathcal{O}_K$, and $x \in \mathcal{V}_{\mathfrak{P}} \cap \ker [N]$. Then one of the following holds:
\begin{enumerate}
    \item $\log N(\mathfrak{P}) \leq CN^{\frac{2n-1}{n(n-1)}}$, or
    \item there exists $y \in \mathcal{V}_\eta$ such that $y$ is contained in an algebraic subgroup of $\mathbb{G}^n_{m,K}$ of codimension at least $2$ and $x$ is contained in the Zariski closure of $y$ in $\mathcal{V}$.
\end{enumerate}

In both cases, we have that
\[\log N(\mathfrak{P}) \leq \begin{cases} CN^{\frac{3}{2}} &\text{if } n = 2,\\
CN &\text{if } n \geq 3,
\end{cases}\]
unless $y$ can be chosen inside $\ker [N]$.
\end{thm}

In less formal or precise terms, one should think of the second alternative as: ``$x$ is obtained as the reduction of $y$ modulo $\mathfrak{P}$ and $y$ arises from an unlikely intersection in the generic fiber".

The set of $y \in \mathcal{V}_\eta$ that are contained in an algebraic subgroup of $\mathbb{G}^n_{m,K}$ of codimension at least $2$ is finite by \cite[Th\'eor\`eme 1.2]{Maurin08}. If $x$ is the reduction of such a $y$, it is probably not possible in general to obtain an upper bound for $\log N(\mathfrak{P})$ that is better than linear in $N$: if for example $K=\mathbb{Q}$, $n = 3$, and $\mathcal{V}$ is defined by the equations $Y - (X+1) = Z - (X+3) = 0$ in the coordinates $X,Y,Z$ on $\mathbb{G}^3_{m,\mathbb{Z}}$, then for any Mersenne prime $p = 2^N-1$, the reduction of the point $(1,2,4) \in \mathcal{V}(\mathbb{Q})$ modulo $p$ has order dividing $N$.

\begin{proof}
We let $\overline{K}$ denote a fixed algebraic closure of $K$ and we let $\overline{\mathbb{F}}_{\mathfrak{P}}$ denote a fixed algebraic closure of $\mathbb{F}_{\mathfrak{P}}$. We set $G = G_N = \{z \in \overline{\mathbb{F}}_{\mathfrak{P}};\, z^N = 1\} \subseteq \overline{\mathbb{F}}_{\mathfrak{P}}^{\ast}$. This is a cyclic group of order dividing $N$ and we fix a generator $\xi \in G$. Every point $x$ as in the theorem is the image of an $\overline{\mathbb{F}}_{\mathfrak{P}}$-point $(x_1,\hdots,x_n) \in (\ker [N])(\overline{\mathbb{F}}_{\mathfrak{P}})$. We deduce that there exist $u_i \in \mathbb{N}$ such that $x_i = \xi^{u_i}$ ($i = 1,\hdots,n$).

We now set $q_1 = \lfloor N^{\frac{1}{n}}\rfloor$ and consider the $(q_1+1)^{n}$ integers $\sum_{i=1}^{n}{r_iu_i}$, $r_i \in \{0,\hdots,q_1\}$. Since $(q_1+1)^n > N$, two of them have to be equal modulo $N$ by the pigeon-hole principle. By taking their difference, we deduce that there exist $k_1,\hdots,k_n \in \{-q_1,\hdots,q_1\}$, not all zero, such that $\prod_{i=1}^{n}{x_i^{k_i}} = 1$. After renumbering the coordinates, we can assume without loss of generality that $k_1 \neq 0$.

Set $\mathcal{H} = \ker(\prod_{i=1}^{n}{X_i^{k_i}})$. By construction, we have that $x \in \mathcal{V} \cap \mathcal{H}$. We have $\mathcal{V} \not\subseteq \mathcal{H}$ since the generic fiber of $\mathcal{V}$ is not contained in any proper algebraic subgroup of $\mathbb{G}^n_{m,K}$. By Proposition \ref{prop:inductivestep} applied with $\mathcal{W} = \mathcal{V}$, there exist a constant $C_1 = C_1(K,n) \geq 1$ and an integral closed subscheme $\mathcal{Z} \subsetneq \mathcal{V}$ such that $x \in \mathcal{Z}$ and
\[ h(\mathcal{Z})+\deg \mathcal{Z} \leq C_1(h(\mathcal{V}) + \deg \mathcal{V})\mathcal{C}(\mathcal{H}) \leq C_1(h(\mathcal{V}) + \deg \mathcal{V})q_1.\]
We deduce that there exists a constant $C_2 = C_2(K,n,\mathcal{V})$ such that
\begin{equation}\label{eq:heightupperboundforz}
    h(\mathcal{Z})+\deg \mathcal{Z} \leq C_2q_1.
\end{equation}

We distinguish two cases: first, $\mathcal{Z}$ could be contained in $\mathbb{G}^n_{m,\mathbb{F}_{\mathfrak{P}}}$. It then follows from Remark \ref{rmk:heightinclosedfiber} and $\deg \mathcal{Z} \geq 1$ that $\log N(\mathfrak{P}) \leq h(\mathcal{Z})$. This proves the theorem thanks to \eqref{eq:heightupperboundforz} and $q_1 \leq N^{\frac{1}{n}}$. We can therefore assume that $\mathcal{Z}$ dominates $\Spec \mathcal{O}_K$. In particular, $\dim \mathcal{Z} = 1$.

Set $q_2 = \lfloor N^{\frac{1}{n-1}}\rfloor$. By a similar argument with the pigeon-hole principle as above, we find $l_2,\hdots,l_n \in \{-q_2,\hdots,q_2\}$, not all zero, such that $\prod_{i=2}^{n}{x_i^{l_i}} = 1$. Set $\mathcal{H}' = \ker(\prod_{i=2}^{n}{X_i^{l_i}})$.

If $\mathcal{Z} \subseteq \mathcal{H}'$, then $\mathcal{Z}_\eta$ is contained in an algebraic subgroup of $\mathbb{G}^n_{m,K}$ of codimension at least $2$. Moreover, since $\mathcal{Z}$ is integral, we have that $\mathcal{Z}_\eta = \{y\}$ for some $y \in \mathcal{V}_\eta$ and so we are done.

Otherwise, every irreducible component of $\mathcal{Z} \cap \mathcal{H}'$ is properly contained in $\mathcal{Z} \subsetneq \mathcal{V}$ and so we must have $\dim(\mathcal{Z} \cap \mathcal{H}') = 0$ by Lemma \ref{lem:dimandcodim}~(3). By construction, $x \in \mathcal{Z} \cap \mathcal{H}'$ and so $\{x\}$ is an irreducible component of $\mathcal{Z} \cap \mathcal{H}'$.

As above, Proposition \ref{prop:inductivestep} implies that there exists an integral closed subscheme $\mathcal{Z}' \subsetneq \mathcal{Z}$ such that $x \in \mathcal{Z}'$ and
\[ h(\mathcal{Z}')+\deg \mathcal{Z}' \leq C_1(h(\mathcal{Z}) + \deg \mathcal{Z})\mathcal{C}(\mathcal{H}') \leq C_1(h(\mathcal{Z}) + \deg \mathcal{Z})q_2.\]
It follows from $\dim \mathcal{Z} = 1$ and Lemma \ref{lem:dimandcodim}~(3) that $\mathcal{Z}' = \{x\}$.

Combining the above inequality with \eqref{eq:heightupperboundforz}, we obtain that
\[ h(\{x\}) \leq C_1C_2q_1q_2.\]
It then follows from Remark \ref{rmk:heightinclosedfiber} and $\deg \{x\} \geq 1$ that
\[ \log N(\mathfrak{P}) \leq h(\{x\}) \leq C_1C_2q_1q_2.\]
Since $q_1q_2 \leq N^{\frac{1}{n}+\frac{1}{n-1}} = N^{\frac{2n-1}{n(n-1)}}$, everything apart from the last paragraph of the theorem now follows.

To prove the last paragraph of the theorem, we can assume that there exists $y \in \mathcal{V}_{\eta}$ such that $y$ is contained in an algebraic subgroup of $\mathbb{G}^n_{m,K}$ of codimension at least $2$ and $x$ is contained in the Zariski closure of $y$ in $\mathcal{V}$. By \cite[Th\'eor\`eme 1.2]{Maurin08}, the set $\Sigma$ of all $z \in \mathcal{V}_{\eta}$ that are contained in an algebraic subgroup of $\mathbb{G}^n_{m,K}$ of codimension at least $2$ is finite. The upper bound then follows from applying Theorem \ref{thm:LIB} to the Zariski closure of $z$ in $\mathcal{V}$ for each $z$ in $\Sigma$.
\end{proof}

\subsection{Two open problems that follow from Conjecture \ref{conj:mainconj}}\label{sec:openproblems}

The following problems are particular instances of Conjecture \ref{conj:mainconj} with specific choices of $\mathcal{V}$ and $\mathcal{H}$.

\vspace{0.2cm}

\textbf{Problem 1:} Let $\mathcal{V} \subseteq \mathbb{G}^3_{m,\mathbb{Z}}$ be defined by $Y - (X+1) = Z - (X-1) = 0$ so that $\dim \mathcal{V} = 2$. In this case, the generic fiber $\mathcal{V}$ is not contained in a proper algebraic subgroup of $\mathbb{G}^3_{m,\mathbb{Q}}$.

Set $\mathcal{H} = \mathcal{H}_N = \ker [N]$ for varying $N \in \mathbb{N}$. Then $\dim(\mathcal{V} \cap \mathcal{H}) = 0$. Lemma \ref{lem:complexitymultbyN} and Remark \ref{rmk:lhszerodim} show that proving Conjecture \ref{conj:mainconj} for $\mathcal{V}$ and $\mathcal{H}_N$ is equivalent to proving the following statement: for any $\varepsilon > 0$, there exists a constant $c = c(\varepsilon)$ such that for all $N \in \mathbb{N}$, the cardinality of the ring
\[ A_N = \mathbb{Z}[T]/(T^N-1,(T+1)^N-1,(T-1)^N-1)\]
satisfies $\log |A_N| \leq \varepsilon N^2 + c$.

\vspace{0.2cm}

\textbf{Remarks on Problem 1:} The sequence $(|A_N|)_{N \in \mathbb{N}}$ is a divisibility sequence as, if $N\mid M$, there is a natural surjection $A_M \twoheadrightarrow A_N$. During a hiking trip in the vicinity of Banff, Koymans asked the following question: can the ring $A_N$ be trivial for infinitely many choices of $N$? The answer to this question is ``no". The reason is the following: set $\rho = \frac{1+\sqrt{5}}{2}$. We have that $\rho+1 = \rho^2$ and $\rho-1 = \rho^{-1}$. Therefore, we have a surjective homomorphism
\[ A_N \twoheadrightarrow \mathbb{Z}[\rho]/(\rho^N-1,(\rho-1)^N-1,(\rho+1)^N-1) = \mathbb{Z}[\rho]/(\rho^N-1)\]
given by $T\mapsto \rho$. If this last ring is trivial, we deduce that $\rho^N - 1 \in \mathbb{Z}[\rho]^{\ast}=\langle -1 , \rho \rangle$ and hence $\rho^N-1 = \pm \rho^m$ for some $m \in \mathbb{Z}$. Since $\rho > 1 $, it follows that $\rho^N-1 = \rho^m$ and $m < N$. If $N > 2$, then we obtain the following chain of inequalities
\[ \rho^N = \rho^m+1 \leq \rho^{N-1} + 1 < \rho^{N-1}+\rho^{N-2} = \rho^N\]
and hence a contradiction. Thus, $A_N$ is non-trivial for all $N > 2$.

Here, $(\rho,\rho+1,\rho-1) \in \mathcal{V}(\overline{\mathbb{Q}})$ is a point whose coordinates satisfy two independent multiplicative relations. There are at most finitely many such points, see \cite{CohenZannier_2000} for a list of all of them (their coordinates are all algebraic integers). One can ask whether, for infinitely many $N$, all points of $\mathcal{V} \cap \mathcal{H}_N$ arise as the reduction of one of these finitely many points in characteristic $0$. We do not know the answer to this question. A positive answer would constitute an analogue of \cite[Conjecture A]{AilonRudnick} and \cite[Conjecture 10]{Silverman_2005}.

Finally, taking $\mathcal{V}$ and $\mathcal{H}_N$ as in Problem 1, one can show that the exponent $\dim \mathcal{V} = 2$ in Conjecture \ref{conj:mainconj} cannot be replaced by $1$, not even with a constant $C$ in place of $\varepsilon$: let $N = p-1$ for some prime $p > 2$, then
\[ (\mathcal{V} \cap \mathcal{H}_N) \times_{\Spec \mathbb{Z}} \Spec \mathbb{F}_p \simeq \Spec \mathbb{F}_p^{p-3}\]
and so the left-hand side in Conjecture \ref{conj:mainconj} is greater than or equal to $(p-3)\log p$ as $p \to \infty$. A similar argument with $\mathcal{V}' \subseteq \mathbb{G}^n_{m,\mathbb{Z}}$ defined by $X_2 - (X_1+1) = X_3 - (X_1-1) = 0$ shows that in general the exponent $\dim \mathcal{V}' = n-1$ cannot be replaced by $\dim \mathcal{V}'-1$.

For computational data related to Problem 1, see Appendix \ref{appendix:data}.

\vspace{0.2cm}

\textbf{Problem 2:} Let $a \in \mathbb{Z} \backslash \{0,\pm 1\}$ and let $\mathcal{V} \subseteq \mathbb{G}^3_{m,\mathbb{Z}}$ be defined by $Z - a = Y - (X+1) = 0$ so that $\dim \mathcal{V} = 2$. Set $\mathcal{H} = \mathcal{H}_{M,N} = \ker(X^M, Y^M, Z^N)$ for varying $M, N \in \mathbb{N}$. Then $\dim(\mathcal{V} \cap \mathcal{H}) = 0$.

Remark \ref{rmk:lhszerodim} shows that proving Conjecture \ref{conj:mainconj} for $\mathcal{V}$ and $\mathcal{H}_{M,N}$ would imply the following statement: for any $\varepsilon > 0$, there exists a constant $c = c(a,\varepsilon)$ such that for all $M, N \in \mathbb{N}$, the cardinality of the ring
\[ R_{M,N} = (\mathbb{Z}/(a^N-1)\mathbb{Z})[T]/(T^M-1,(T+1)^M-1)\]
satisfies $\log |R_{M,N}| \leq \varepsilon \max\{M,N\}^2 + c$. We invite the reader to compare this problem with Proposition \ref{prop:Lfunction}.

\section*{Acknowledgements}

We thank Fabrizio Barroero, Ziyang Gao, Philipp Habegger, Peter Koymans, Lars K\"uhne, David Masser, Pieter Moree, Fabien Pazuki, Nicolas Perrin, Jonathan Pila, Christian Urech, and Umberto Zannier for helpful comments and discussions. We thank Lars K\"uhne for pointing out the connection between LIB and the work of Dimitrov-Habegger, we thank Michael Stoll for pointing out the connection between the resultant of two polynomials in $\mathbb{Z}[T]$ and the cardinality of the ring obtained by quotienting out the ideal generated by them, and we thank Denis Simon and Robert Wilms for assistance with the computational exploration of Problem 1 in Section \ref{sec:openproblems} and, in Wilms's case, for agreeing to co-author an appendix presenting the data he obtained. We thank Iris Simon for correcting our citation of J. G. Fichte and for explanations on the phenomenology of German idealism.

The first-named author was supported by ANR-20-CE40-0003 Jinvariant and by the AAP IMPULSION 2026 of the Université Clermont Auvergne. He is grateful to Max Planck Institute for Mathematics in Bonn for its hospitality and financial support. When this project began, the second-named author was supported by the Early Postdoc.Mobility grant no. P2BSP2\_195703 of the Swiss National Science Foundation. He thanks the Mathematical Institute of the University of Oxford and his host there, Jonathan Pila, for hosting him as a visitor for the duration of this grant. We warmly thank the organizers of the conferences ``Leuca2022" and ``La grandezza dei punti piccoli" for providing a stimulating environment for our collaboration.

This material is based upon work supported by the National Science Foundation under Grant No. DMS--1928930 while the second-named author was in residence at the Mathematical Sciences Research Institute in Berkeley, California, during the Spring 2023 semester. The second-named author thanks the DFG for its support (grant no.\ EXC-2047/1 - 390685813). Both authors are grateful for support through the PHC 2026\_16 Germaine de Sta\"el project DiPARATroL.

\vspace{\baselineskip}
\noindent
\framebox[\textwidth]{
\begin{tabular*}{0.96\textwidth}{@{\extracolsep{\fill} }cp{0.84\textwidth}}
		\raisebox{-0.7\height}{%
			\begin{tikzpicture}[y=0.80pt, x=0.8pt, yscale=-1, inner sep=0pt, outer sep=0pt, 
			scale=0.12]
			\definecolor{c003399}{RGB}{0,51,153}
			\definecolor{cffcc00}{RGB}{255,204,0}
			\begin{scope}[shift={(0,-872.36218)}]
			\path[shift={(0,872.36218)},fill=c003399,nonzero rule] (0.0000,0.0000) rectangle (270.0000,180.0000);
			\foreach \myshift in 
			{(0,812.36218), (0,932.36218), 
				(60.0,872.36218), (-60.0,872.36218), 
				(30.0,820.36218), (-30.0,820.36218),
				(30.0,924.36218), (-30.0,924.36218),
				(-52.0,842.36218), (52.0,842.36218), 
				(52.0,902.36218), (-52.0,902.36218)}
			\path[shift=\myshift,fill=cffcc00,nonzero rule] (135.0000,80.0000) -- (137.2453,86.9096) -- (144.5106,86.9098) -- (138.6330,91.1804) -- (140.8778,98.0902) -- (135.0000,93.8200) -- (129.1222,98.0902) -- (131.3670,91.1804) -- (125.4894,86.9098) -- (132.7547,86.9096) -- cycle;
			\end{scope}
			\end{tikzpicture}%
		}
		&
		Francesco Campagna and Gabriel A. Dill have received funding from the European Research Council (ERC) under the European Union's Horizon 2020 research and innovation programme (grant agreement n$^\circ$ 945714).
	\end{tabular*}
}

\appendix

\section{Computational data for Problem 1 -- joint with Robert Wilms}\label{appendix:data}

For $N \in \mathbb{N}$, let
\[ A_N = \mathbb{Z}[T]/(T^N-1,(T+1)^N-1,(T-1)^N-1)\]
as in Problem 1 and set $a_N = |A_N|$.

If the ring
\[\mathbb{Z}[T]/(T^N-1,(T-1)^N-1+s((T+1)^N-1))\]
is finite for some $s \in \mathbb{Z}$, then $a_N$ divides its size, which, by \cite[Lemma 1]{Frenkel_Zabradi_2018}, is equal to the resultant of $T^N-1$ and $(T-1)^N-1+s((T+1)^N-1)$ (in \cite{Frenkel_Zabradi_2018}, both polynomials are assumed to be monic, but the same proof works if only one of them is monic). Building on this, we developed C++ code to verify that 
\[ \log a_N \leq 4N\log N\]
for all $N \leq 3095$.

This supports the conjecture that $\log a_N = o(N^2)$ and even supports the stronger claim that $\log a_N = \mathcal{O}(N \log N)$, which is best possible as one can see by taking $N = p-1$ (see Section \ref{sec:openproblems} above).

We show here two figures that illustrate the generated data; the first one shows an upper bound for $(\log a_N)/(N \log N)$ (for $2 \leq N \leq 3095$) while the second one shows the exact value of $(\log a_N)/(N \log N)$ (for $562$ values of $N$ between $2$ and $586$). In both figures, the values for prime $N$ appear in red while the values for $N$ divisible by $q-1$ for powers $q = p^k > 3$ of at least $10$ distinct primes $p$ appear in yellow.

\begin{figure}[h]
\includegraphics[width=0.9\textwidth]{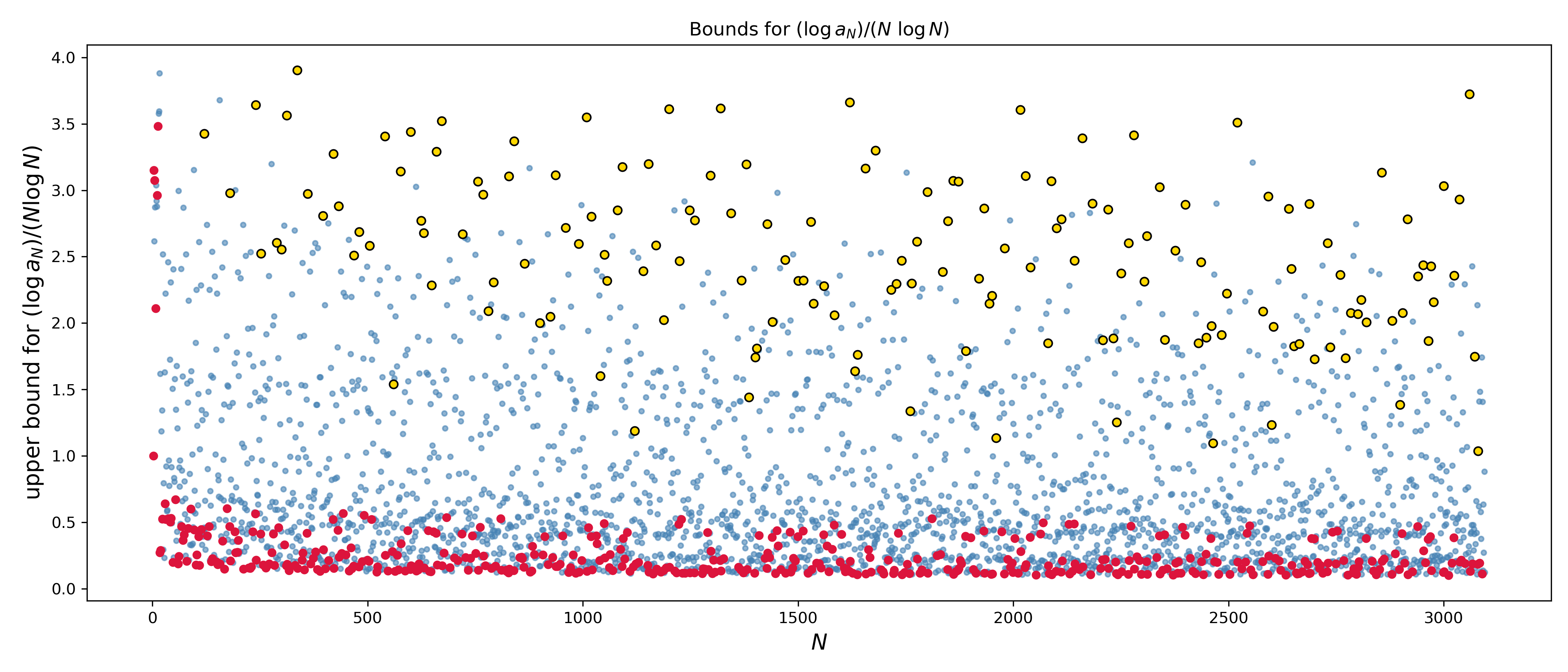}
\end{figure}

\begin{figure}[h]
\includegraphics[width=0.9\textwidth]{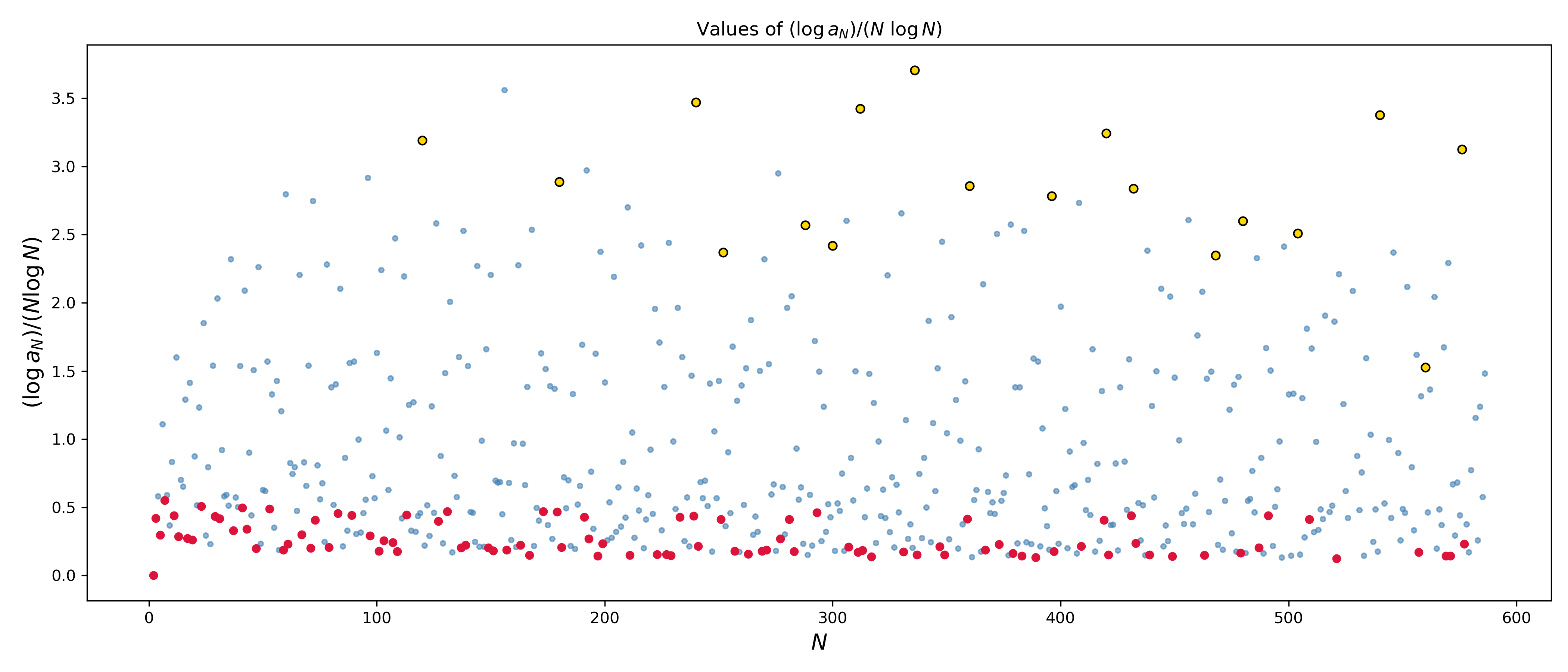}
\end{figure}

The code that generates the upper bound for $a_N$ computes several multiples of $a_N$, corresponding to different randomly chosen integers $s$, and then determines their greatest common divisor. Further values of $s$ are included until the logarithm of the resulting greatest common divisor is at most $4N\log N$. In most cases, two values of $s$ were sufficient. Occasionally, a third one was required, whereas for small $N$, a single value of $s$ sometimes already sufficed. This also explains the high red points near the left-hand edge of the figure showing the upper bounds.

One might expect from a divisibility sequence that the values $a_p$ for $p$ prime tend to be among the smallest ones. Moreover, one would expect $a_N$ to be large if $N$ is divisible by $q-1$ for powers $q = p^k > 3$ of many primes $p$. This is supported by the numerical data, even if one considers $(\log a_N)/(N\log N)$ instead of $a_N$ itself. For instance, for $N = 336$, which is divisible by $q-1$ for $q = 8$, $9$, $17$, $25$, $29$, $43$, $49$, $113$, $169$, $337$, and $N = 3060$, which is divisible by $q-1$ for $q = 5$, $7$, $11$, $13$, $19$, $31$, $37$, $61$, $103$, $181$, $256$, $307$, $613$, $1021$, $1531$, $3061$, the value of and the upper bound for $(\log a_N)/(N \log N)$ respectively are among the largest ones computed.

\printbibliography

\end{document}